\documentclass[final]{siamart250211}

\usepackage{algorithm, algpseudocode, algorithmicx}
\usepackage{amsopn, amsmath, amscd, amssymb, bm, mathrsfs, mathtools}
\usepackage{array, booktabs, makecell} 
\newcolumntype{P}[1]{>{\centering\arraybackslash}p{#1}}
\newcolumntype{M}[1]{>{\centering\arraybackslash}m{#1}}
\usepackage[american]{babel}
\usepackage[ddmmyy,hhmmss]{datetime}
\usepackage{enumerate}
\usepackage{graphicx}
\usepackage{pdflscape}
\usepackage{hyperref}
\usepackage{url}
\usepackage{setspace}
\usepackage{soul}
\usepackage[backgroundcolor=violet!60,textcolor=white,linecolor=gray]{todonotes}

\usepackage{textcomp}

\usepackage[mathscr]{euscript}

\let\trace\relax
\DeclareMathOperator{\trace}{trace}
\DeclareMathOperator{\diag}{diag}

\DeclareMathOperator{\spn}{span}	

\DeclareMathOperator{\gsstep}{\texttt{gs\_step}}
\DeclareMathOperator{\bgsstep}{\texttt{bgs\_step}}
\newcommand{\qreq}{\overset{\mbox{\normalfont\tiny\sffamily QR}}{=}}

\usepackage{xspace}		

\newcommand{\spEK}{\mathscr{E}}				
\newcommand{\spK}{\mathscr{K}}				
\newcommand{\spQK}{\mathscr{Q}}				
\newcommand{\spR}{\mathbb{R}}				

\newcommand{\vb}{\bm{b}}

\newcommand{\ve}{\bm{e}}

\newcommand{\vq}{\bm{q}}
\newcommand{\vr}{\bm{r}}
\newcommand{\vrhat}{\widehat{\vr}}
\newcommand{\vs}{\bm{s}}

\newcommand{\vv}{\bm{v}}

\newcommand{\vw}{\bm{w}}
\newcommand{\vx}{\bm{x}}

\newcommand{\vB}{\bm{B}}

\newcommand{\vQ}{\bm{Q}}

\newcommand{\vS}{\bm{S}}

\newcommand{\vV}{\bm{V}}

\newcommand{\vW}{\bm{W}}
\newcommand{\vX}{\bm{X}}

\newcommand{\vY}{\bm{Y}}

\newcommand{\Hunder}{\underline{H}}

\newcommand{\Tunder}{\underline{T}}

\newcommand{\VV}{\mathcal{V}}

\newcommand{\bVV}{\bm{\mathcal{V}}}

\newcommand{\inv}{{-1}}

\newcommand{\fAb}{f(A)\vb}

\newcommand{\ip}[2]{\langle #1,#2 \rangle}

\newcommand{\norm}[1]{\left\lVert#1\right\rVert}

\makeatletter
\newsavebox{\@brx}
\newcommand{\llangle}[1][]{\savebox{\@brx}{\(\m@th{#1\langle}\)}%
	\mathopen{\copy\@brx\kern-0.5\wd\@brx\usebox{\@brx}}}
\newcommand{\rrangle}[1][]{\savebox{\@brx}{\(\m@th{#1\rangle}\)}%
	\mathclose{\copy\@brx\kern-0.5\wd\@brx\usebox{\@brx}}}
\makeatother

\newcommand{\gl}{\mbox{\normalfont\tiny\sffamily gl}}
\newcommand{\cl}{\mbox{\normalfont\tiny\sffamily cl}}
\newcommand{\IPgl}[2]{{\llangle #1,#2 \rrangle}_{\gl}}
\newcommand{\IPcl}[2]{{\llangle #1,#2 \rrangle}_{\cl}}

\newcommand{\bigO}{\mathcal{O}}

\newcommand{\M}{{\mbox{\normalfont\tiny\sffamily M}}}
\renewcommand{\S}{{\mbox{\normalfont\tiny\sffamily S}}}

\ifpdf
  \DeclareGraphicsExtensions{.eps,.pdf,.png,.jpg}
\else
  \DeclareGraphicsExtensions{.eps}
\fi

\newsiamremark{remark}{Remark}
\newsiamremark{hypothesis}{Hypothesis}
\crefname{hypothesis}{Hypothesis}{Hypotheses}
\newsiamthm{claim}{Claim}
\newsiamremark{fact}{Fact}
\crefname{fact}{Fact}{Facts}

\newsiamremark{example}{Example}

\Crefname{subsection}{Section}{Sections}

\headers{Stability of basis construction in extended Krylov methods}{K.~Lund \& M.~Schweitzer}

\title{On the stability of basis construction in extended Krylov subspace methods\thanks{Submitted to the editors \today.}
\funding{London Mathematical Society, Research in Pairs program, Reference: 42415; DFG project 566049107, ``The $f(A)b$ulous scientific network on matrix functions and exponential integrators''}}

\author{Kathryn Lund\thanks{Computational Mathematics Theme, Building R71, STFC Rutherford Appleton Laboratory, Harwell Oxford, Didcot, Oxfordshire, OX11 0QX, United Kingdom. \email{kathryn.lund@stfc.ac.uk}}
\and Marcel Schweitzer\thanks{School of Mathematics and Natural Sciences, Bergische Universit\"at Wuppertal, 42097 Wuppertal, Germany.  \email{marcel@uni-wuppertal.de}}
}

\ifpdf
\hypersetup{
  pdftitle={On the stability of basis construction in extended Krylov subspace methods},
  pdfauthor={K.~Lund, M.~Schweitzer}
}
\fi

\begin{document}

\maketitle

\begin{abstract}
    Extended Krylov subspace methods (EKSMs) are a well known approach for approximating the action of a matrix function on a vector $f(A)\mathbf{b}$, and their block versions are often used for solving Lyapunov and Sylvester matrix equations.  EKSMs produce Laurent approximations to the underlying operator, which are generally more accurate per iteration than polynomial KSMs, without requiring a pole selection procedure as rational KSMs do.  At the same time, due to the expansion of the search space by not only $A$ but also by its inverse, a number of implementation choices must be made that can affect the stability of the algorithm in floating-point arithmetic.  We concentrate on two variants that often appear in the literature and in code, based on whether new basis vectors are built via what we denote as ``mixed channels" or ``split channels."  We characterize their loss of orthogonality with an easily computable metric and prove that when $A$ is symmetric positive definite, mixed channels have better orthogonality than split channels; however, numerical experiments demonstrate that straightforward predictions are not possible for indefinite or nonsymmetric matrices.  We also provide a modular MATLAB toolbox, \texttt{ektoolbox}, that contains new versions of EKSMs with low-synchronization Gram-Schmidt backbones.
\end{abstract}

\begin{keywords}
    extended Krylov subspace methods, matrix function times a vector, rounding error, numerical stability, loss of orthogonality
\end{keywords}

\begin{MSCcodes}
    65F25, 65G50, 65F60, 65F10
\end{MSCcodes}

\section{Introduction} \label{sec:intro}

Krylov subspace methods (KSMs) are ubiquitous in the solution of linear systems, matrix functions, matrix equations, and eigenvalue solvers.  For matrix functions in particular, rational Krylov subspace methods (RKSMs) have proven to be highly efficient and accurate, especially relative to polynomial KSMs, assuming there is a straightforward method for obtaining suitable poles \cite{Gut10}.  As finding optimal poles and solving a sequence of shifted linear systems can themselves be computationally challenging tasks, it is common to look to extended Krylov subspace methods (EKSMs) as a balance between polynomial and rational KSMs.  Furthermore, in the special case that the operator is unitary, extended Krylov methods are natural to implement, as inverting the operator is trivial \cite{RinV26} (while the same is not true for shifted unitary operators).  

Given a square matrix $A \in \spR^{n \times n}$ and a vector $\vb \in \spR^n$, a (polynomial) Krylov subspace of dimension $m$ is defined as
\[
\spK_m(A, \vb) = \{\vb, A \vb, A^2\vb, \ldots, A^{m-1}\vb\}.
\]
Meanwhile a rational Krylov subspace is defined as
\[
\spQK_m(A, \vb) := q(A) \spK_m(A, \vb),
\]
for a rational function $q(z) = \prod_{i = 1}^p \frac{1}{z - \xi_i}$ with poles $\xi_i \in \spR$, $i = 1, \ldots, p$.  EKSMs force $\xi_i \in \{0, \infty\}$, for all $i = 1, \ldots, p$.  We define the extended Krylov subspace of orders $(\ell, m)$ as
\begin{equation} \label{eq:ek_subspace}
    \spEK_{\ell,m}(A, \vb) := \spn\{A^{-\ell+1}\vb, \ldots, A^{-2}\vb, A^\inv \vb, \vb, A\vb, A^2\vb, \ldots, A^{m-1}\vb\},
\end{equation}
which comprises $\ell + m - 1$ basis vectors. In practice one may want fewer inverses than powers of $A$ (or vice versa), and indeed this has been investigated by Jagels and Reichel \cite{JagR09, JagR11}. To narrow the scope of this investigation, we will assume the number of applications of each operator is the same, i.e., $\ell = m$.

The space $\spEK_{m,m}(A, \vb)$ takes on a rather unique shape relative to either $\spK_m(A,\vb)$ or $\spQK_m(A,\vb)$ and requires unique implementation decisions, such as the following:
\begin{itemize}
    \item In what order should powers of $A$, respectively $A^{-1}$, be added?
    \item Should vectors be added in batches, or one at a time?
    \item How should the projection of $A$ onto the subspace be computed?
    \item How should linear systems with $A$ be solved (and how accurately)?
\end{itemize}
We analytically address the first two of these points by considering the effects of rounding error; we will numerically investigate the third.  The fourth point remains for future work but necessarily depends on the outcome of the previous questions.  To simplify the analysis for now, we will assume that we have access to a highly accurate LU, LDL, or Cholesky factorization of $A$ that can be stored and easily accessed from memory.

We classify EKSM variants as either \emph{mixed channel} or \emph{split channel} depending on how new vectors are generated.  The mixed channel approach applies $A$ (respectively $A^{-1}$) to the most recently generated basis vector.  Meanwhile the split channel approach applies $A^{-1}$ to the last generated odd vector and $A$ to last even vector (or the other way around, as long as the same operator acts exclusively on odd or even vectors).  The split channel approach further allows for batching the generation of new vectors, i.e., using block Gram-Schmidt on the $A$ and $A^{-1}$ directions simultaneously.

Viewing EKSMs through these lenses leads to a number of interesting results.  We begin by categorizing variants in the literature in \Cref{sec:variants}.  To kick off the analysis in \Cref{sec:stability}, we show that mixed- and split-channel approaches not only span the same space but produce essentially the same basis, up to signs (\Cref{thm:same_span}); furthermore, for symmetric $A$, mixed approaches have a 3-term recurrence while split approaches satisfy a 5-term recurrence (\Cref{thm:short_recurrences}).  Through an easily computable metric, we prove that, regardless of the choice of orthogonalization backbone, mixed-channel methods have predictably better orthogonality than split-channel methods for symmetric positive definite $A$ (\Cref{thm:spd_mcil_vs_scil}).  Numerical experiments in \Cref{sec:experiments} are carried out through a modular MATLAB toolbox, \texttt{ektoolbox}, which contains configurations of EKSMs not previously implemented, such as low-synchronization Gram-Schmidt backbones and some block variants.  These experiments demonstrate that it is not possible to predict which approach is better for $A$ indefinite or nonsymmetric, especially when the error of computing $f(A)\vb$ is considered alongside the loss of orthogonality.  We summarize and outline future directions in \Cref{sec:conclusions} and collect technical proofs and lemmas supporting the main theorems in \Cref{appendix:proofs}.

We employ standard numerical linear algebra notation throughout: uppercase Roman letters ($A, H, K$) denote matrices; lowercase Roman letters denote matrix entries ($a_{ij}, h_{ij}, k_{ij}$); bold Roman letters denote vectors or concatenations of vectors ($\vv$ or $\vV$, respectively); and the occasional Greek letter denotes special constants ($\beta, \gamma, \rho, \sigma$).  Unless otherwise stated, we use the Euclidean inner product $\ip{\cdot}{\cdot}$ and 2-norm $\norm{\cdot}$.  All in-text code can be executed in MATLAB.

\section{Implementation variants} \label{sec:variants}
EKSMs were introduced by Druskin and Knizhnerman for computing Markov functions $f$ of symmetric matrices \cite{DruK98}.  As RKSMs were already established in the literature, EKSMs were regarded as a special case in which only the poles $0$ and $\infty$ are considered, effectively leading to the construction of a subspace with alternating applications of an operator $A$ and its inverse. Since then, EKSMs have been used to compute other matrix functions \cite{KniS10}, and to solve Lyapunov and Sylvester matrix equations \cite{Sim07, Hey10, KniS11}, families of shifted linear systems \cite{AlG25, Sim10}, and parametric problems \cite{PalTNetal25, Sim10}.  EKSMs have been thoroughly investigated for symmetric positive definite matrices, revealing connections to orthogonal Laurent polynomials and short-term recurrences \cite{MerV15, JagR09, JagR11}.  Block variants have also been proposed \cite{BenEJ20, BenJS17, SadBSetal20}.  Convergence analysis has also been addressed, albeit without accounting for floating-point error \cite{GocG13, Sch15, Sch16}.

Building an extended Krylov subspace means constructing an orthonormal basis $\vV_{2m-1} \in \spR^{n \times 2m-1}$ spanning \eqref{eq:ek_subspace} and coefficient matrices $K_{2m-1} \in \spR^{2m-1 \times 2m-1}$ and $\Hunder_{2m} \in \spR^{2m \times 2m-1}$ such that an extended Arnoldi relation holds:
\begin{equation} \label{eq:ek_arnoldi_hk}
    A \vV_{2m-1} K_{2m-1} = \vV_{2m} \Hunder_{2m},
    \mbox{ where }
    \Hunder_{2m} = \begin{bmatrix} H_{2m-1} \\ h_{2m,2m-1} \ve_{2m-1}^T \end{bmatrix}.
\end{equation}
Other treatments of EKSMs focus instead on an Arnoldi relation of a more familiar form:
\begin{equation} \label{eq:ek_arnoldi_t}
    A \vV_{2m-1} = \vV_{2m} \Tunder_{2m},
    \mbox{ where }
    \Tunder_{2m} = \begin{bmatrix} T_{2m-1} \\ t_{2m,2m-1} \ve_{2m-1}^T \end{bmatrix}.
\end{equation}
We focus on the coefficients $\Hunder_{2m}$ and $K_{2m-1}$ separately for the moment, but look at different ways $\Tunder_{2m}$ can be constructed, as well as that construction's effect on numerical stability in \Cref{sec:experiments}.

As the order of vectors in the span~\eqref{eq:ek_subspace} is irrelevant from a purely mathematical point of view, there are a number of choices to make when constructing the relation~\eqref{eq:ek_arnoldi_hk}.  First of all, we assume that the first orthogonalization step is performed on $\vb$ alone, i.e., $\vv_1 \beta = \vb$, where $\vv_1$ has norm $1$, and $\beta = \norm{\vb}$.  For the next step, we must decide
\begin{itemize}
    \item whether $A$ or $A^\inv$ is applied first (i.e., is the ``leading" operator), and
    \item whether operators only act on basis vectors in their own ``channel", or on the most recently computed one.
\end{itemize}
These options lead to four variants:
\begin{table}[H]
    \centering
    \begin{tabular}{c|c|c}
                        & power-leading & inverse-leading   \\ \hline
         mixed channel  & MCPL          & MCIL              \\ \hline
         split channel  & SCPL          & SCIL
    \end{tabular}
    \caption{Table of four possible extended Krylov implementation variants}
    \label{tab:ek_variants}
\end{table}

In the following, we will make use of the shorthand $\vq r \qreq \vw$, to signify that $\norm{\vq} = 1$, $r = \norm{\vw}$, and $\vq = \vw/r$.  We will also make use of a subroutine $\gsstep$ (``Gram-Schmidt step") such that $[\vq,\, r,\, \vs] = \gsstep(\vV, \vw)$ is equivalent to
\[
    \vq \, r \qreq \vw - \vV \vs, \mbox{ and } \vs = \vV^T \vw.
\]
Similarly, $[\vQ,\, R,\, \vS] = \bgsstep(\bVV, \vW)$ (``Block Gram-Schmidt step") is equivalent to
\[
    \vQ \, R \qreq \vW - \bVV \vS, \mbox{ and } \vS = \bVV^T \vW,
\]
where $\qreq$ denotes an economic QR decomposition with the same dimensions as $\vW$.  For the moment, we do not specify the details of $\bgsstep$ but note that it could be implemented by a variety of block Gram-Schmidt routines \cite{CarLRetal22, Lun23}.  Note that (block) Lanczos variants trivially fall into one of the categories of \Cref{tab:ek_variants}, as Lanczos is mathematically equivalent to Gram-Schmidt.

Whether $A$ or $A^\inv$ is applied first does not strongly influence the resulting numerical stability of the algorithm.  All the same, we briefly look at the two different approaches in the following two sections.

\subsection{Power-leading variants} \label{sec:power_leading}
We start by examining a mixed-channel, power-leading (MCPL) approach.  The first basis vector is obtained by normalizing $\vb$: $\vv_1 \beta \qreq \vb$.  Next, we apply $A$ to $\vv_1$ and apply a $\gsstep$ to obtain $\vv_2$:
\begin{equation} \label{line:mcpl-v2}
    [\vv_2,\, h_{21},\, h_{11}] = \gsstep(\vv_1, A \vv_1) \quad\Leftrightarrow\quad A \vv_1  = \vV_2 H_{1:2,1} 
\end{equation}
To obtain $\vv_3$, we now apply $A^\inv$ to $\vv_2$ and another $\gsstep$:
\begin{equation} \label{line:mcpl-v3}
	[\vv_3,\, k_{32},\, K_{1:2,2}] = \gsstep(\vV_2, A^\inv \vv_2) \quad\Leftrightarrow\quad A \vV_3 K_{1:3,2} = \vv_2 
\end{equation}
Combining \eqref{line:mcpl-v2} and \eqref{line:mcpl-v3} into a single, Arnoldi-like relation leads to
\begin{equation} \label{eq:mcpl-ek-arnoldi}
	A \vV_3 \underbrace{\begin{bmatrix}
		1 	& k_{12} \\
			& k_{22} \\
			& k_{32}
	\end{bmatrix}}_{=: \underline{K}_3}
	=
	\vV_2 \underbrace{\begin{bmatrix}
		h_{11}	&	\\
		h_{21}	& 1
	\end{bmatrix}}_{=: H_2}.
\end{equation}
The relation~\eqref{eq:mcpl-ek-arnoldi} is perhaps better viewed as an Arnoldi relation for the inverse operator:
\begin{equation*}
	A^{\inv} \vV_2 H_2 = \vV_3 \underline{K}_3,
\end{equation*}
While this formulation may be useful for some applications where relations in $A^\inv$ are required, we will not consider power-leading variants further in the present work.
    
\subsection{Inverse-leading variants} \label{sec:inverse_leading}
Constructing relations for inverse-leading variants is similar to the process in \Cref{sec:power_leading}.  In lieu of the derivation, we provide \Cref{alg:mcil-ek-arnoldi} and \Cref{alg:scil-ek-arnoldi}, which summarize the Arnoldi procedure for MCIL and SCIL, respectively.  Equations~\eqref{eq:K-mcil-ek-arnoldi} and \eqref{eq:H-mcil-ek-arnoldi} show the general structure of $K_{2m-1}$ and $\Hunder_{2m}$ for MCIL.  Note that empty space represents zeros, and both matrices are upper Hessenberg.

\begin{equation}
    K_{2m-1}
    =
    \begin{bmatrix}
        k_{11}  &   & k_{13}    &   & k_{15}    &       & \cdots    & k_{1,2m-2}    &   \\
        k_{21}  & 1 & k_{23}    &   & k_{25}    &       &           &               &   \\
                &   & k_{33}    &   & k_{35}    &       &           &               &   \\
                &   & k_{43}    & 1 & k_{45}    &       &           &               &   \\
                &   &           &   & k_{55}    &       &           & \vdots        &   \\
                &   &           &   & k_{65}    & 1     &           &               &   \\
                &   &           &   &           & \ddots& \ddots    &               &   \\
                &   &           &   &           &       & \ddots    & k_{2m-2,2m-2} &   \\
                &   &           &   &           &       &           & k_{2m-1,2m-2} & 1
    \end{bmatrix}  \label{eq:K-mcil-ek-arnoldi}
\end{equation}
\begin{equation}
    \Hunder_{2m}
    =
    \begin{bmatrix}
        1   & h_{12}  &   & h_{14}    &   & h_{16}  & \cdots    & \cdots    & h_{1,2m-1}  \\
            & h_{22}  &   & h_{24}    &   & h_{26}  &           &           &           \\
            & h_{32}  & 1 & h_{34}    &   & h_{36}  &           &           &           \\
            &         &   & h_{44}    &   & h_{46}  &           &           &           \\
            &         &   & h_{54}    & 1 & h_{56}  &           &           & \vdots    \\
            &         &   &           &   & h_{66}  &           &           &           \\
            &         &   &           &   & h_{76}  & \ddots    &           &           \\
            &         &   &           &   &         & \ddots    & 1         &           \\
            &         &   &           &   &         &           & \ddots    & h_{2m-1,2m-1} \\
            &         &   &           &   &         &           &           & h_{2m,2m-1}
    \end{bmatrix} \label{eq:H-mcil-ek-arnoldi}
\end{equation}

\begin{algorithm}[htbp!]
	\caption{MCIL-EK-Arnoldi} \label{alg:mcil-ek-arnoldi}
	\begin{algorithmic}[1]
		\State{Given $A$, $\vb$, $m$}
		\State{$\vv_1 \beta \qreq \vb$}
		\For{$j = 1, \ldots, m$}
    		\State{$\vw_{j.1} = A^\inv \vv_{2j-1}$} \label{line:mcil-w-j.1}
    		\State{$[\vv_{2j},\, k_{2j,2j-1},\, K_{1:2j-1,2j-1}] = \gsstep(\vV_{2j-1}, \vw_{j.1})$} \label{line:mcil-gsstep-j.1}
    		\State{$h_{2j-1,2j-1} = 1$}
    		\State{$\vw_{j.2} = A \vv_{2j}$} \label{line:mcil-w-j.2}
    		\State{$[\vv_{2j+1},\, h_{2j+1,2j},\, H_{1:2j,2j}] = \gsstep(\vV_{2j}, \vw_{j.2})$}  \label{line:mcil-gsstep-j.2}
    		\State{$k_{2j,2j} = 1$}
		\EndFor
		\State \Return{$\vV_{2m}, K_{2m-1}, \Hunder_{2m}$ satisfying \eqref{eq:ek_arnoldi_hk}}
	\end{algorithmic}
\end{algorithm}

\Cref{alg:scil-ek-arnoldi} shows how to ``split" the channels, in the sense that, after the first three basis vectors are generated, $A^\inv$ is only ever applied to basis vectors of even index, and $A$ to those of odd index.  Equations~\eqref{eq:K-scil-ek-arnoldi} and \eqref{eq:H-scil-ek-arnoldi} demonstrate what $K_{2m-1}$ and $\Hunder_{2m}$ look like for SCIL.

\begin{equation}
    K_{2m-1}
    =
    \begin{bmatrix}
        k_{11}  & 1 & k_{13}    &   & k_{15}    &       & \cdots    & k_{1,2m-2}    &   \\
        k_{21}  &   & k_{23}    &   & k_{25}    &       &           &               &   \\
                &   & k_{33}    & 1 & k_{35}    &       &           &               &   \\
                &   & k_{43}    &   & k_{45}    &       &           & \vdots        &   \\
                &   &           &   & k_{55}    & 1     &           &               &   \\
                &   &           &   & k_{65}    & \ddots& \ddots    &               &   \\
                &   &           &   &           & \ddots& \ddots    & k_{2m-3,2m-2} &   \\
                &   &           &   &           &       & \ddots    & k_{2m-2,2m-2} & 1 \\
                &   &           &   &           &       &           & k_{2m-1,2m-2}   & 
    \end{bmatrix}  \label{eq:K-scil-ek-arnoldi}
\end{equation}
\begin{equation}
\Hunder_{2m}
    =
    \begin{bmatrix}
        1   & h_{12}  &   & h_{14}    &   & h_{16}  & \cdots    & \cdots    & h_{1,2m-1}    \\
            & h_{22}  & 1 & h_{24}    &   & h_{26}  &           &           &               \\
            & h_{32}  &   & h_{34}    &   & h_{36}  &           &           &               \\
            &         &   & h_{44}    & 1 & h_{46}  &           &           & \vdots        \\
            &         &   & h_{54}    &   & h_{56}  &           &           &               \\
            &         &   &           &   & h_{66}  & \ddots    &           &               \\
            &         &   &           &   & h_{76}  & \ddots    & 1         & h_{2m-3,2m-1} \\
            &         &   &           &   &         & \ddots    & \ddots    & h_{2m-2,2m-1} \\
            &         &   &           &   &         &           & \ddots    & h_{2m-1,2m-1} \\
            &         &   &           &   &         &           &           & h_{2m,2m-1}
    \end{bmatrix} \label{eq:H-scil-ek-arnoldi}
\end{equation}

\begin{algorithm}[htbp!]
	\caption{SCIL-EK-Arnoldi} \label{alg:scil-ek-arnoldi}
	\begin{algorithmic}[1]
        \State{Given $A$, $\vb$, $m$}
		\State{$\vv_1 \beta \qreq \vb$}
		\State{$\vW_1 = \begin{bmatrix} A^\inv \vv_1 & A \vv_1 \end{bmatrix}$} \label{line:scil-w-1}
		\State{$\left[ [\vv_2,\, \vv_3],\,
		    \begin{bmatrix}
		        k_{21}  & h_{22} \\
		                & h_{32}
		    \end{bmatrix},\,
		    \begin{bmatrix}
		        k_{11} & h_{12}
		    \end{bmatrix} \right] = \bgsstep(\vv_1, \vW_1)$} \label{line:scil-bgsstep-1}
		\State{$k_{12} = 1$}
		\State{$h_{11} = 1$}
		\For{$j = 2, \ldots, m-1$}
    		\State{$\vW_j = \begin{bmatrix} A^\inv \vv_{2j-2} & A \vv_{2j-1} \end{bmatrix}$} \label{line:scil-w-j}
    		\State{$\left[ [\vv_{2j},\, \vv_{2j+1}],\,
    		    \begin{bmatrix}
    		        k_{2j,2j-1}   & h_{2j,2j} \\
    		                    & h_{2j+1,2j}
    		    \end{bmatrix},\,
    		    \begin{bmatrix}
    		        K_{1:2j-1,2j-1} & H_{1:2j-1,2j}
    		    \end{bmatrix} \right] = \bgsstep(\vV_{2j-1}, \vW_j)$} \label{line:scil-bgsstep-j}
    		\State{$k_{2j-1,2j} = 1$}
    		\State{$h_{2j-1,2j-1} = 1$}
		\EndFor
		\State \Return{$\vV_{2m-1}, K_{2m-1}, \Hunder_{2m}$ satisfying \eqref{eq:ek_arnoldi_hk}}
	\end{algorithmic}
\end{algorithm}

Observe that the structures of $K_{2m-1}$ and $\Hunder_{2m}$ for both MCIL and SCIL are very similar: in both cases, $K_{2m-1}$ and $\Hunder_{2m}$ are upper Hessenberg matrices with even columns being standard unit vectors. Furthermore, the first column of $K_{2m-1}$ is identical for both variants, as well as the first two basis vectors $\vv_1$ and $\vv_2$.  These observations motivate \Cref{thm:same_span}.

\begin{theorem} \label{thm:same_span}
    Let $\vV_{2m}^\M$ and $\vV_{2m}^\S$ denote the bases generated by \Cref{alg:mcil-ek-arnoldi,alg:scil-ek-arnoldi}, respectively, assuming no breakdowns (i.e., no subdiagonal elements are zero).  Then not only do $\vV_{2m}^\M$ and $\vV_{2m}^\S$ both span the EKS~\eqref{eq:ek_subspace}, but also $\vv_i^\M$ and $\vv_i^\S$ are identical up to sign for all $i \in \{1, \ldots, 2m-1\}$, i.e.,
    \begin{equation} \label{eq:same_span}
    	\vv_i^\M = \vv_i^\S c_i
    \end{equation}
where $c_i = \pm 1$ for all $i \in \{1, \ldots, 2m-1\}$.
\end{theorem}
\begin{proof}
    That \Cref{alg:mcil-ek-arnoldi,alg:scil-ek-arnoldi} produce bases spanning the extended Krylov subspace~\eqref{eq:ek_subspace} is well-known. See, e.g,~\cite{Sim07} for the SCIL case and \cite{JagR09} for the MCIL case.
    
    Clearly, the first two basis vectors $\vv_1, \vv_2$ agree in both algorithms, so that~\eqref{eq:same_span} holds with $c_1 = c_2 = 1$.  Now let $i = 2j+1, j\geq 1$ be odd. Then, by construction, 
    \begin{equation} \label{eq:intersection_spaces}
        \vv_i^\M, \vv_i^\S \in \spEK_{j+1,j+1}(A,\vb) \cap \left(\spEK_{j+1,j}(A,\vb)\right)^\perp.
    \end{equation}
    As we assumed that no breakdown occurs, we have $\dim(\spEK_{j+1,j+1}(A,\vb)) = 2j+1, \dim(\spEK_{j+1,j}(A,\vb)) = 2j$. As further $\spEK_{j+1,j}(A,\vb) \subseteq \spEK_{j+1,j+1}(A,\vb)$, the space on the right-hand side of~\eqref{eq:intersection_spaces} has dimension exactly one. Thus, it directly follows that $\vv_i^\M$ and $\vv_i^\S$ are collinear. As they are of unit norm, the collinearity factor needs to satisfy $c_i = \pm 1$.
    For $i = 2j+2, j \geq 1$ odd, we have
    \[
        \vv_i^\M, \vv_i^\S \in \spEK_{j+2,j+1}(A,\vb) \cap \left(\spEK_{j+1,j+1}(A,\vb)\right)^\perp,
    \]
    and we can argue completely analogously to the odd case.
\end{proof}

\begin{remark} \label{rem:basis}
It is quite striking that despite the different candidate vectors and order of operations, \Cref{alg:mcil-ek-arnoldi,alg:scil-ek-arnoldi} essentially produce the same basis for the extended Krylov space. Thus, one might assume that it really does not matter much which algorithmic variant one chooses, and our exploration of the literature and code for EKSMs indicates that many researchers (at least implicitly) apply this logic and assume that all implementations are the same mathematically. However, both our experimental and theoretical findings reported in \Cref{sec:stability,sec:experiments} indicate that this is far from true once floating-point arithmetic enters the picture.
\end{remark}

\begin{remark} \label{rem:hkinv}
    When using an EKSM for approximating the action of a matrix function or the solution of a matrix equation, one requires access to the compression of $A$ onto the approximation space. In light of~\eqref{eq:ek_arnoldi_t}, we obviously have $\vV_{2m-1}^TA\vV_{2m-1} = T_{2m-1}$. In contrast to polynomial Krylov methods, where the compressed matrix simply contains the coefficients from the Gram--Schmidt orthogonalization and is thus available without any additional computations, one does not obtain $T_{2m-1}$ ``for free'' in extended Krylov methods.
    
    By considering the special structure of the matrices $H_{2m-1}$ and $K_{2m-1}$, cf.~\eqref{eq:K-mcil-ek-arnoldi}--\eqref{eq:H-mcil-ek-arnoldi} or \eqref{eq:K-scil-ek-arnoldi}--\eqref{eq:H-scil-ek-arnoldi}, it is easy to see that $T_{2m-1}$ is upper triangular with two additional subdiagonals, and that it is therefore pentadiagonal when $A$ is symmetric.
    
    There are at least three possible approaches for forming $T_{2m-1}$: 
    \begin{itemize}
        \item \texttt{explicit}: explicitly from the relation $T_{2m-1} = \vV_{2m-1}^T A \vV_{2m-1}$, which, however, is not attractive from a computational point of view, as it requires many additional matrix-vector products, inner products, and memory operations;
        \item \texttt{hk}: as $T_{2m-1} = H_{2m-1} K_{2m-1}^\inv$ by directly inverting $K_{2m-1}$ via, e.g., Gaussian elimination; or
        \item \texttt{direct}: by exploiting the special structure of the matrices $H_{2m-1}, K_{2m-1}$ to derive recurrences for directly computing the entries of $T_{2m-1}$. This has been done, e.g., in~\cite{Sim07} in the nonsymmetric setting and is specialized for the symmetric case in~\cite{JagR11}.
    \end{itemize}
    
    The different ways of constructing $T_{2m-1}$ behave differently in floating-point arithmetic, which we investigate in \Cref{sec:experiments}.
\end{remark}

\begin{remark} \label{rem:block}
    It is of course possible to consider a block vector $\vB$ of size $n \times s$ instead of the column vector $\vb$.  Deriving block implementations of EKSMs is relatively straightforward: in fact, Simoncini's extension of~\cite{DruK98} is already a block method~\cite{Sim07}.  For SCIL, care must be taken with the block size in lines~\ref{line:scil-bgsstep-1} and \ref{line:scil-bgsstep-j} of \Cref{alg:scil-ek-arnoldi}, which will be $2s$ instead of just $s$.  Block versions are implemented in our code \texttt{ektoolbox} and discussed in \Cref{sec:experiments}, but their careful analysis remains for future work.
\end{remark}

Similarly to the Lanczos process in the context of polynomial Krylov methods, it is well-known that the extended Krylov basis vectors fulfill a short recurrence when $A$ is symmetric. Interestingly, although both methods produce essentially the same basis according to \Cref{thm:same_span}, the MCIL approach leads to a three-term recurrence, while the SCIL approach leads to a five-term recurrence.

We summarize this observation in the following theorem. Although several variants of this result can be found in the literature, they do not exactly fit our setting. We therefore think it is instructive to also include a proof.

\begin{theorem} \label{thm:short_recurrences}
    Let $A$ be symmetric. Then, for the quantities from \Cref{alg:mcil-ek-arnoldi}, we have
    \begin{equation} \label{eq:short_recurrences_mcil}
        A^\inv \vv^\M_{2j-1} = \sum_{i=2j-2}^{2j} k^\M_{i,2j-1} \vv^\M_i, \qquad A\vv^\M_{2j} = \sum_{i=2j-1}^{2j+1} h^\M_{i,2j} \vv^\M_i,
    \end{equation}
    while for the quantities from \Cref{alg:scil-ek-arnoldi}, we have
    \begin{equation} \label{eq:short_recurrences_scil}
        A^\inv\vv^\S_{2j-2} = \sum_{i=2j-4}^{2j} k^\S_{i,2j-1} \vv^\S_i, \qquad A\vv^\S_{2j-1} = \sum_{i=2j-3}^{2j+1} h^\S_{i,2j} \vv^\S_i,
    \end{equation}
\end{theorem}
\begin{proof}
    We start with the MCIL approach. Clearly, from \Cref{alg:mcil-ek-arnoldi}, the basis vectors fulfill the recurrences
    \begin{equation*}
        A^\inv \vv^\M_{2j-1} = \sum_{i=1}^{2j} k^\M_{i,2j-1} \vv^\M_i, \qquad A \vv^\M_{2j} = \sum_{i=0}^{2j+1} h^\M_{i,2j} \vv^\M_i,
    \end{equation*}
    with
    \begin{equation} \label{eq:coefficients_mcil}
        k^\M_{i,2j-1} = \ip{\vv^\M_i}{A^\inv \vv^\M_{2j-1}}, \qquad h^\M_{i,2j} = \ip{\vv^\M_i}{A \vv^\M_{2j}}
    \end{equation}
    and we will argue that most of the coefficients~\eqref{eq:coefficients_mcil} are zero. Due to the symmetry of $A$, we can rewrite~\eqref{eq:coefficients_mcil} as
    \begin{equation*}
        k^\M_{i,2j-1} = \ip{A^\inv \vv^\M_i}{\vv^\M_{2j-1}}, \qquad h^\M_{i,2j} = \ip{A \vv^\M_i}{\vv^\M_{2j}}
    \end{equation*}
    and notice that for odd $i = 2\ell-1$, $\vv_i^\M \in \spEK_{\ell,\ell}(A,\vb) \setminus \spEK_{\ell,\ell-1}(A,\vb)$, so that $A^\inv\vv_i^\M \in \spEK_{\ell+1,\ell-1}(A,\vb)$ and $A\vv_i^\M \in \spEK_{\ell-1,\ell+1}(A,\vb)$. Now, by construction, $\vv_{2j-1}^\M$ is orthogonal to $\spEK_{j,j-1}(A,\vb)$ and $\vv_{2j}^\M$ is orthogonal to $\spEK_{j,j}(A,\vb)$. Thus, $\ip{A^\inv \vv^\M_i}{\vv^\M_{2j-1}} = 0$ and $\ip{A \vv^\M_i}{\vv^\M_{2j}} = 0$ if $j \geq \ell+1$, i.e., if $i \leq 2j-3$. 
    
    Similarly, in the even case $i = 2\ell$, we have $\vv_i^\M \in \spEK_{\ell+1,\ell}(A,\vb) \setminus \spEK_{\ell,\ell}(A,\vb)$, so that $A^\inv\vv_i^\M \in \spEK_{\ell+2,\ell-1}(A,\vb)$ and $A\vv_i^\M \in \spEK_{\ell,\ell+1}(A,\vb)$. Again exploiting the orthogonality relations for $\vv_{2j-1}^\M$ and $\vv_{2j}^\M$, we have that $\ip{A^\inv\vv^\M_i}{\vv^\M_{2j-1}} = 0$ if $j \geq \ell+2$, i.e., if $i \leq 2j-4$ and $\ip{A\vv^\M_i}{\vv^\M_{2j}} = 0$ if $j \geq \ell +1$, i.e., if $i \leq 2j-2$. 
    
    Combining both cases, $k^\M_{i,2j-1} = 0$ for $i \leq 2j-3$ and $h^\M_{i,2j} = 0$ for $i \leq 2j-2$, which proves~\eqref{eq:short_recurrences_mcil}. 
    
    For SCIL, we start analogously, writing
    \begin{equation*}
        A^\inv \vv^\S_{2j-2} = \sum_{i=1}^{2j} k^\S_{i,2j-1} \vv^\S_i, \qquad A \vv^\S_{2j-1} = \sum_{i=0}^{2j+1} h^\S_{i,2j} \vv^\S_i,
    \end{equation*}
    with
    \begin{equation*}
        k^\S_{i,2j-1} = \ip{A^\inv \vv^\S_i}{\vv^\S_{2j-2}}, \qquad h^\S_{i,2j} = \ip{A \vv^\S_i}{\vv^\S_{2j-1}}
    \end{equation*}
    and proceed as before.
    
    For odd $i = 2\ell-1$, due to $A^\inv \vv_i^\M \in \spEK_{\ell+1,\ell-1}(A,\vb)$ and $A \vv_i^\M \in \spEK_{\ell-1,\ell+1}(A,\vb)$, we have $\ip{A^\inv \vv^\S_i}{\vv^\S_{2j-2}} = 0$ and $\ip{A \vv^\S_i}{\vv^\S_{2j-1}} = 0$ if $j - 1 \geq \ell+1$, i.e., if $i \leq 2j-5$, because $\vv^\S_{2j-2}$ is orthogonal to $\spEK_{j-1,j-1}(A,\vb)$ and $\vv^\S_{2j-1}$ is orthogonal to $\spEK_{j,j-1}(A,\vb)$.
    
    For even $i = 2\ell$, we have $A^\inv \vv_i^\M \in \spEK_{\ell+2,\ell-1}(A,\vb)$ and $A \vv_i^\M \in \spEK_{\ell,\ell+1}(A,\vb)$ and from the above orthogonality properties, it follows that $\ip{A^\inv \vv^\S_i}{\vv^\S_{2j-2}} = 0$ if $j-1 \geq \ell+2$, i.e., if $i \leq 2j-6$ and $\ip{A \vv^\S_i}{\vv^\S_{2j-1}} = 0$ if $j-1 \geq \ell+1$, i.e., if $i \leq 2j-5$. 
    
    By combining both cases, we find $k^\S_{i,2j-1} = 0$ for $i \leq 2j-5$ and $h^\S_{i,2j} = 0$ for $i \leq 2j-4$, which proves~\eqref{eq:short_recurrences_scil} and thus concludes the proof. 
\end{proof}

\subsection{Categorization of variants in the literature}
In order to make fair comparisons across EKSMs, it is important to understand how the numerical results have been generated.  For example, in Lund's own previous work \cite{KreLMetal21}, the MCIL version was used as a control method when testing a newer approach on both symmetric and nonsymmetric problems.  Although the test problems were well-behaved enough that it appears no major instability occurred, more challenging examples could have led to bias against the control method.  In this section, we discuss the implementation choices for a few select publications and classify these, along with others, in \Cref{tab:categorization}.

Druskin and Knizhnerman are the first to formalize EKSMs for matrix functions \cite{DruK98}.  At first glance, their proposed algorithm (detailed in \cite[Section~5]{DruK98} for Lanczos-based orthogonalization) appears to fit in neither the SCIL nor the MCIL paradigms.  They begin by computing a basis of size $m$ for the space generated by $A^\inv$ and then feed the first basis vector $\vv_1$ into an Arnoldi/Lanczos process for $A$.  Although bases for the parts corresponding to $A^\inv$ and $A$ are computed sequentially, because $\vv_1 \beta \qreq \vb$, the end effect is that the operators act on separate channels.  We denote this variant ``SeqSCIL", for ``sequential SCIL".

Some years later, Simoncini adapted the EKSM of \cite{DruK98} to a solver for Lyapunov matrix equations, commonly known as K-PIK (Krylov-plus-inverted-Krylov) \cite{Sim07}.  Although the operations are reordered and batched into blocks, Simoncini still enforces that $A$ and $A^\inv$ are applied to odd and even channels, respectively.  The reformulation is not exactly the same as our SCIL, as Simoncini's routine is initialized with the block vector $\begin{bmatrix} \vb & A^\inv \vb \end{bmatrix}$, which leads to one extra power of $A^\inv$ in the end.  Despite these differences, we classify \cite{Sim07} as SCIL.

Jagels and Reichel analyze EKSM from the perspective of orthogonal Laurent polynomials in \cite{JagJR21, JagR09, JagR11}.  They examine and provide pseudocode adapted to scenarios where $\ell \neq m$.  In each of their provided algorithms, it appears that $A$ and $A^\inv$ are in fact generated by mixing channels. They focus on Lanczos formulations for Hermitian $A$ and derive recurrence formulas for computing coefficients of the compressed matrix. Similarly, Schweitzer derives a non-Hermitian extended Lanczos method (employing matrix-vector products and linear solves with both $A$ and $A^T$) in~\cite{Sch17}, which also features a five-term recurrence for the basis vectors and recurrence relations for the coefficients. In contrast to the work by Jagels and Reichel, this work uses an SCIL approach. A block-version of the nonsymmetric extended Lanczos method (also using an SCIL approach) is considered in~\cite{SadA21} in the context of differential Lyapunov equations. 

Heyouni revisits extended block Arnoldi for Sylvester equations and uses the same process as K-PIK, including batching $\vB$ with $A^\inv \vB$ for the first step  \cite{Hey10}.  The block algorithm is further generalized to allow for global inner products, i.e., where the ``classical" block inner product $\IPcl{\vX}{\vY} := \vX^T\vY$ is replaced by
\[
    \IPgl{\vX}{\vY} := \frac{1}{s}\trace\big(\vX^T\vY\big) I_s,
\]
for $\vX, \vY \in \spR^{n \times s}$, and $I_s$ denoting the $s \times s$ identity.  The global inner product was first introduced in \cite{BouJSetal06}.

G\"ockler and Grimm treat a more general version of EKSM, wherein $(A-\gamma I)^\inv$ is used instead of $A^\inv$, for a fixed shift $\gamma \in \spR$ \cite{GocG13}\footnote{Such an approach was already advocated earlier in~\cite{BecR09}, albeit without a discussion of algorithmic details.}. Their pseudocode indicates they compute an Arnoldi relation of dimension $q$ in $A$ first, and then extend the basis via $(A-\gamma I)^\inv$, with $\vv_q$ initializing the process.  We refer to this variant as SeqMCIL. Similarly, Xu and Xue discuss the use of a pole different from 0 in~\cite{XuX23} in the context of approximating Stieltjes (or Markov) functions, while~\cite{GocG13} focuses on $\varphi$-functions in exponential integrators. The treatment in~\cite{XuX23} does not include a detailed algorithmic description in the paper itself, but the authors shared their code with Lund, who identified it as SCIL. In~\cite{ShaS13}, a similar approach is used for constrained Sylvester equations, where it is necessary to work with a shifted inverse, as the original, unshifted matrix is not invertible. As basis generation is not described in detail, it is not clear whether an MCIL or SCIL approach is used.

In~\cite{Sch16}, Schweitzer proves that extended Krylov approximations for $f(A)\vb$, where $f$ is a Laplace--Stieltjes function and $A$ is a Hermitian matrix, converge strictly monotonically, in the sense that the norm of the error reduces with each iteration of the method. The proof of this result makes use of information about the sign pattern in the compressed matrix $H_{2m-1}K_{2m-1}^\inv$ and builds upon recurrence relations proven in~\cite{JagR09,JagR11}. Further,~\cite[Section~7]{Sch15} presents computable \textit{a posteriori} error estimates for extended Krylov methods. The discussion is in terms of an SCIL approach, but the results can be straightforwardly generalized to MCIL.

Mertens and Vandebril establish a position vector that records the order in which new vectors are added \cite{MerV15}.  Their approach is general enough that it could be used to describe any of the variants in \Cref{tab:categorization}; indeed, they note that their ``algorithm has the potential to become adaptive and give the user the opportunity to adjust the position vector at runtime in order to improve convergence.  However, the question of how to choose an optimal position vector is a difficult one\dots''  Their default ordering appears to be MCPL. 

\begin{table}[htbp!]
    \centering
    \begin{tabular}{l|p{10.5cm}}
    	\thead{Variant}	& \thead{Source} \\ \hline
    	SeqSCIL			& \cite{DruK98} \\
    	SCIL			& \cite{Abd24, BenEJ20, BenJS17, BenSSetal18, BenEJetal24, BreSS16, BurHK12, ChaFCetal24, ChaGFetal20, DuZCetal24, EscS25, Flo24, FloES20, GiaGEetal24, GutS21, HacJ18a, HamJR25, Hey10, KniS10, NobA23, PalTNetal25, SadA21, SadA25, SadBSetal20, Sch15, Sch17, Sim07, Sim10, StoGFetal22, StoGFetal25, TseAFetal24} \\
    	SeqMCIL			& \cite{GocG13}	\\
    	MCIL			& \cite{AlG25, BenGJetal18, JagJR21, KreLMetal21, XuX23}\\
        MCPL            & \cite{JagR09, JagR11, MerV15, Sch16} \\
    	SCPL			& none \\
        Others          & \cite{GriL21}$^\dagger$, \cite{BenFS22}$^\ast$ \\
        Unknown         & \cite{ElTQetal24, ShaS13, XuX23}
    \end{tabular}
    \caption{Categorization of which EKSM variant a given source uses. ${}^\ast$This source uses non-standard first steps due to the Hamiltonian structure but after that is pure MC. $\dagger$This is actually more like a shift-invert method. Only a single multiplication with $A$ is performed, then only inverses are applied.}
    \label{tab:categorization}
\end{table}

See \Cref{tab:categorization} for a summary of the discussions in this section.  We have also categorized many sources that were not explicitly discussed but were straightforward to categorize.  Most implementations we have found follow a split-channel approach.

\section{On the loss of orthogonality} \label{sec:stability}

To the best of our knowledge, the floating-point stability of any variant of EKSM has not yet been explored.  An interesting quote from \cite{GocG13} seems to summarize the prevailing \emph{modus operandi}: ``Obvious improvements with respect to stability and efficiency [of the EKSM] should be performed."  Arguably, such improvements are not so obvious, and perhaps even the obvious ones could hinder accuracy or performance.

A first, somewhat simplified approach to understanding the loss of orthogonality (LOO) in the computed basis is the following: in each step of both algorithms, a candidate vector $\vw$ is orthogonalized against a given subspace $\VV$ and then normalized to obtain the next basis vector $\vq$. That is, in exact
arithmetic,
\[
    \vq = \frac{\vr}{\norm{\vr}}, \qquad \vr = (I - \Pi_{\VV}) \vw,
\]
where $\Pi_{\VV}$ is the orthogonal projector onto $\VV$. Using a standard model of floating-point arithmetic \cite{Hig02}, one actually computes
\[
    \vrhat = \vr + \delta\vr, \qquad \norm{\delta\vr} \lesssim cu \norm{\vw},
\]
where $u$ is the unit round-off and $c$ is a moderate constant. Thus, the normalization of the basis vectors amplifies the perturbation via
\[
    \frac{\norm{\delta\vr}}{\norm{\vr}}
    \lesssim cu \frac{\norm{\vw}}{\norm{\vr}}
    = cu \frac{\norm{\vw}}{\norm{(I-\Pi_{\VV})\vw}}.
\]
In other words, whenever the norm of $(I-\Pi_{\VV})\vw$ (the vector remaining after orthogonalization) is very small compared to the norm of the candidate vector $\vw$, a substantial error amplification might take place, which in turn leads to a substantial LOO. We therefore investigate the reciprocal quantities
\begin{equation*}
 \rho_k^\S :=\frac{\norm{(I-\Pi_{\VV_k})\vw_k^\S}}{\norm{\vw_k^\S}}, \qquad
  \rho_k^\M :=\frac{\norm{(I-\Pi_{\VV_k})\vw_k^\M}}{\norm{\vw_k^\M}},
\end{equation*}
where $\vw_k^\S$ and $\vw_k^\M$ are the candidate vectors of SCIL and MCIL, respectively, and $\VV_k$ is either of the form $\spEK_{j,j-1}(A, \vb)$ or $\spEK_{j,j}(A, \vb)$, depending on whether $k$ is even or odd. Thus, smaller values of $\rho_k^\S$ or $\rho_k^\M$ correspond to larger local error amplification. Even more compactly, we can use the ratio
\begin{equation} \label{eq:gamma}
  \gamma_k:=\frac{\rho_k^\M}{\rho_k^\S}
\end{equation}
to compare the two algorithmic variants. If $\gamma_k>1$ at any given step, we can expect a larger LOO in the SCIL variant, while $\gamma_k<1$ indicates that MCIL might have a larger LOO.

This point of view does not account for all potential sources of LOO, which might also depend on the specific orthogonalization routine that is used.  For example, the analysis in~\cite[Section~3]{Bjo94} shows that the orthogonality of a basis computed by modified Gram--Schmidt (which is arguably the default for most Krylov implementations) might actually suffer from cancellations in each individual substep (i.e., orthogonalization against one individual basis vector). However, as the following example illustrates, the ratio~\eqref{eq:gamma} is in fact a good proxy for comparing the MCIL and SCIL variants.

\begin{example} \label{ex:nonsymm}
    \begin{figure}
        \centering
        \begin{tabular}{cc}
            \includegraphics[width=0.45\linewidth]{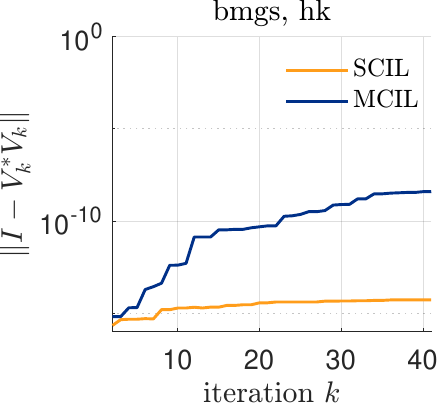} &
            \includegraphics[width=0.45\linewidth]{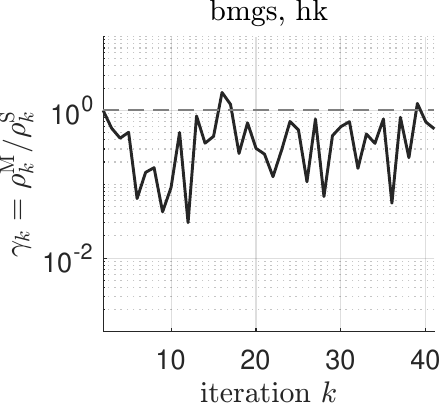}
        \end{tabular}
        \caption{LOO and ``instability ratio'' $\gamma_k$ for $A$ and $\vb$ from \Cref{ex:nonsymm}. The dashed line in the lower plot corresponds to $\gamma_k = 1$, i.e., MCIL and SCIL being on par.}
        \label{fig:ex_stab_nonsymm}
    \end{figure}
    Consider the simple problem setup
    
    \vspace{.2cm}
    
    \begin{verbatim}
        r = rng(4); n = 1000; w = 3;
        M = spdiags(rand(n,2*w+1),-w:w,n,n);
        b = randn(n,1);
    \end{verbatim}

    We run twenty iterations of \Cref{alg:mcil-ek-arnoldi} and \Cref{alg:scil-ek-arnoldi} using modified Gram--Schmidt orthogonalization. \Cref{fig:ex_stab_nonsymm} depicts the LOO of the two extended Krylov bases, as well as the ratios $\gamma_k$ from~\eqref{eq:gamma}. It is clearly visible, particularly in early iterations, that steep increases in the LOO occur exactly when $\gamma_k \ll 1$. We also observe that in general $\gamma_k < 1$ seems to occur much more frequently than $\gamma_k > 1$, and that if $\gamma_k$ exceeds 1, it does so only mildly; see also \Cref{sec:experiments}. This might suggest that one can expect the MCIL approach to sometimes behave a lot worse than SCIL, while the opposite behavior seems rather unlikely.
\end{example}

Interestingly, the observed behavior changes fundamentally once $A$ is symmetric positive definite.

\begin{example} \label{ex:spd}
    \begin{figure}
        \centering
        \begin{tabular}{cc}
            \includegraphics[width=0.45\linewidth]{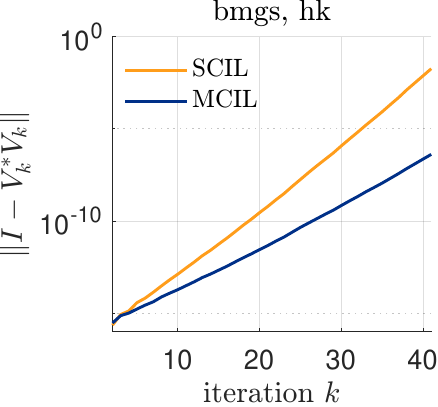} &
            \includegraphics[width=0.45\linewidth]{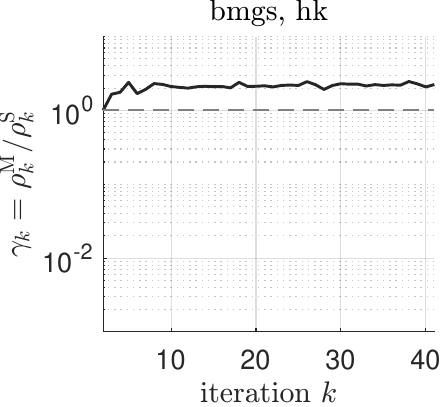} 
        \end{tabular}
        \caption{LOO and ``instability ratio'' $\gamma_k$ for $A$ and $\vb$ from \Cref{ex:spd}. The dashed line in the lower plot corresponds to $\gamma_k = 1$, i.e., MCIL and SCIL being on par.}
        \label{fig:ex_stab_spd}
    \end{figure}
    Consider the following modification of the setup in \Cref{ex:nonsymm}, which makes the matrix $A$ symmetric positive definite.
    
    \vspace{.2cm}
    
    \begin{verbatim}
        r = rng(4); n = 1000; w = 3;
        M = spdiags(rand(n,2*w+1),-w:w,n,n);
        A = M+M'+3.5*speye(n);
        b = randn(n,1);
    \end{verbatim}

    We again run twenty iterations of the MCIL and SCIL extended Krylov method using modified Gram--Schmidt orthogonalization.\footnote{Note that we perform full orthogonalization (despite the existence of short recurrences for the basis vectors -- see \Cref{thm:short_recurrences}) in order to not introduce a further potential source of LOO. Our experiments in \Cref{sec:experiments} show that the qualitative conclusions stay valid also when short recurrences are employed.} \Cref{fig:ex_stab_spd} depicts the LOO of the two extended Krylov bases, as well as the ratios $\gamma_k$ from~\eqref{eq:gamma}. In contrast to before, we now observe that $\gamma_k > 1$ for all $k \geq 3$, and this is also reflected in the LOO, which is consistently larger for SCIL than for MCIL.
\end{example}

In the following, we aim to formalize and explain the observations from \Cref{ex:nonsymm,ex:spd}.  It turns out that in the positive definite case, one obtains a rigorous comparison inequality. The lower plot in \Cref{fig:ex_stab_nonsymm} already shows that in the general case, the comparison between SCIL and MCIL can go either way.

Throughout the remainder of this section, in light of \Cref{thm:same_span}, we assume that the phases are fixed in such a way that SCIL and MCIL produce the common real orthonormal basis $\vv_1,\vv_2,\ldots$. We also use the SCIL QR convention in which the subdiagonal coefficients are positive.

\begin{theorem} \label{thm:spd_mcil_vs_scil}
    Let $A\in\spR^{n\times n}$ be symmetric positive definite, and assume that there is no breakdown in \Cref{alg:mcil-ek-arnoldi,alg:scil-ek-arnoldi}. Then
    \[
        \rho_2^\M = \rho_2^\S \qquad\text{ and }\qquad \rho_k^\M > \rho_k^\S, \qquad k \geq 3.
    \]
    Consequently, with $\gamma_k := \rho_k^\M / \rho_k^\S$, one has $\gamma_2 = 1$ and $\gamma_k > 1$ for all $k \geq 3$.
\end{theorem}

The proof of \Cref{thm:spd_mcil_vs_scil} as well as several technical lemmas it requires are given in \Cref{appendix:proofs}.

    \begin{figure}
        \centering
        \begin{tabular}{cc}
            \includegraphics[width=0.45\linewidth]{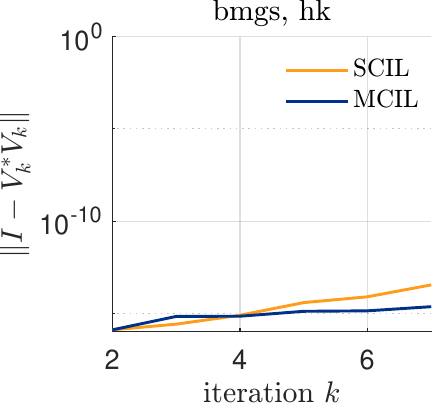} &
            \includegraphics[width=0.45\linewidth]{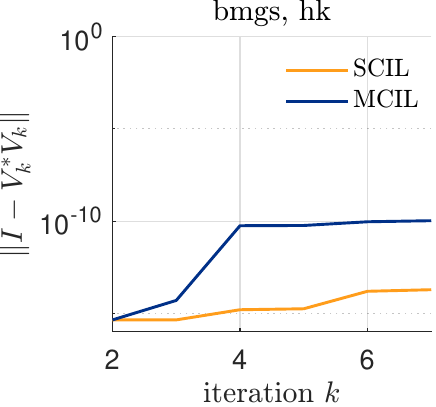} \\
            \includegraphics[width=0.45\linewidth]{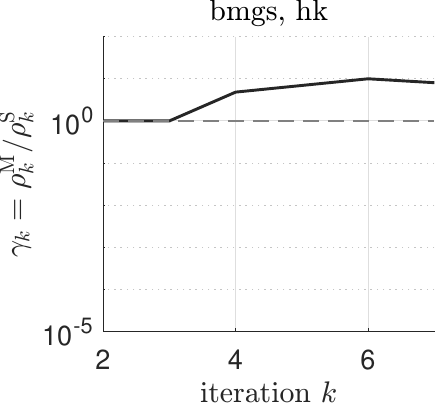} &\includegraphics[width=0.45\linewidth]{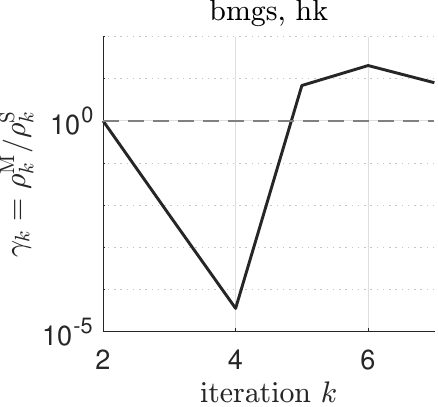} 
        \end{tabular}
        \caption{LOO and ``instability ratio'' $\gamma_k$ for $A$ from \Cref{ex:indefinite} and $\vb^{(1)}=(1,1,1,1,1,1,1,1)^T$ (left) and $\vb^{(2)}=(1,3,3,3,3,3,3,3)^T$ (right). The dashed line in the lower plots corresponds to $\gamma_k = 1$, i.e., MCIL and SCIL being on par.}
        \label{fig:ex_stab_indefinite}
    \end{figure}

\begin{remark} \label{rem:negative_definite}
    \Cref{thm:spd_mcil_vs_scil} also holds when $A$ is \emph{negative} definite, by replacing $A$ with $-A$, which generates essentially the same extended Krylov basis.
\end{remark}

The following example illustrates that \Cref{thm:spd_mcil_vs_scil} crucially depends on the definiteness of $A$. Already for the symmetric indefinite case, one can easily construct matrices for which there is no clear ``ranking'' of the LOO in MCIL and SCIL and where it might even depend on the specific starting vector $\vb$. 
\begin{example} \label{ex:indefinite}
    Consider the symmetric indefinite matrix
    \[
        A = \diag(-10^{-3}, 10^{-2}, 10^\inv, 1, 10, 10^2, 10^3, 10^4) \in \spR^{8 \times 8}
    \]
    and the two vectors
    \[
        \vb^{(1)} = (1,1,1,1,1,1,1,1)^T, \qquad
        \vb^{(2)} = (1,3,3,3,3,3,3,3)^T.
    \]
    \Cref{fig:ex_stab_indefinite} shows the LOO in the bases computed by \Cref{alg:mcil-ek-arnoldi} and \Cref{alg:scil-ek-arnoldi} for $A$ and the two different starting vectors. For $\vb^{(1)}$, $\gamma_k$ is moderately larger than $1$ for all $k$, and in line with that, the SCIL basis shows a slightly larger LOO overall (note however, that for dimension 3, MCIL shows larger LOO despite $\gamma_3 \approx 1$).
    
    For $\vb^{(2)}$, however, the loss in the MCIL basis is much larger. In particular, the fourth step shows a large increase in the LOO, in line with $\gamma_4 \ll 1$. From then on, the LOO in MCIL reaches a plateau, while it gradually increases in SCIL, in line with $\gamma_k > 1$ for $k \geq 5$.
    This example also illustrates again that $\rho_k$ is a \emph{local} indicator for the round-off error amplification occurring in step $k$, but does not say anything about the \emph{accumulated} LOO up to step $k$.
\end{example}

In light of \Cref{ex:indefinite}, it appears difficult to derive any general result akin to \Cref{thm:spd_mcil_vs_scil} for the comparison of MCIL and SCIL in the indefinite (or even nonsymmetric) setting.

\section{Numerical Experiments} \label{sec:experiments}
Experiments in this section demonstrate features of the \texttt{ektoolbox}\footnote{\url{https://doi.org/10.5281/zenodo.23337612}}, designed to prototype EKSM variants for matrix functions acting on a vector or block of vectors.  The toolbox is written modularly in MATLAB (2021b or later), so that users need only specify a few options to implement different variants.

One direct consequence of this modular approach is that several new methods are implemented, namely, EKSM with low-synchronization block Gram-Schmidt backbones.  For the experiments in this section, we look at Lanczos (\texttt{blanczos}), Modified Gram-Schmidt (\texttt{bmgs}), and Classical Gram-Schmidt with Inner Reorthogonalization (\texttt{bcgs\_iro}); single-column methods also pass through the block subroutines, hence the prefix \texttt{b}.  Note that SCIL can use block subroutines regardless of whether the starting vector is itself a block vector.  MCIL can only use block methods if the starting vector is a block vector.

A matrix function $f(A)$ is defined as follows for $f$ analytic on and inside a closed contour $\Gamma$ that encloses the spectrum of $A$:
\[
f(A) = \frac{1}{2 \pi i} \int_{\Gamma} f(z) (zI - A)^\inv dz.
\]
Further equivalent definitions can be found in Higham's monograph \cite{Hig08}.

We consider small ($n = 1000$) matrices that are symmetric positive definite (SPD), symmetric indefinite, or nonsymmetric; see Table~\ref{tab:matrices}.  For contrasting LOO against accuracy of a practical problem, we look at $f(z) = \exp(z)$ and $f(z) = z^{-1/2}$ for random $\vb$ (or $\vB$ with $s = 2$).  We plot the LOO of the computed basis as well as relative error, which is defined with respect to a ground truth solution $f(A)\vb$ computed by built-in MATLAB functions \texttt{expm} or \texttt{sqrtm} with \texttt{backslash} (i.e., dense methods accurate to double precision).  All calculations are carried out in double precision.

\begin{table}[htbp]
  \centering
  \caption{Matrices in the numerical experiments, all built from a shared banded matrix $M\in\spR^{1000\times1000}$ of bandwidth $w=3$ with uniformly sampled random entries in $[0,1]$.}
  \label{tab:matrices}
  \begin{tabular}{lll}
    \toprule
    Matrix & Definition & Properties \\
    \midrule
    $A_1$  & $M + M^{T} + 3.5\,I$  & symmetric positive definite\\
    $A_2$  & $-M - M^{T} + 3.5\,I$ & symmetric indefinite\\
    $A_3$  & $-M$                  & nonsymmetric\\
    $A_4$  & $2M + 3.5\,I$         & nonsymmetric with SPD symmetric part\\
    \bottomrule
  \end{tabular}
\end{table}

\begin{example}[$\exp(A_1) \vb$] \label{ex:expm_spd} This example clearly demonstrates the conclusions of \Cref{thm:spd_mcil_vs_scil}: \Cref{fig:expm_loo_spd} shows that SCIL has worse orthogonality than MCIL, regardless of the method for computing $T$ or the orthogonalization backbone.  (Indeed, here and throughout, it appears that computing $T$ via the \texttt{hk} or \texttt{direct} strategies are identical.  Sharper examples that tease out different behavior remain a topic for future work.)  When \texttt{bcgs\_iro} is used, all methods attain machine precision ($\bigO(10^{-16})$) LOO, so hardly any difference is detected there.  A block size $s > 1$ introduces further instability, so both methods struggle with block Lanczos.

When it comes to computing $f(A)\vb$, however, we see in \Cref{fig:expm_true_err_spd} that the $T$ method has a clear effect---the \texttt{explicit} approach via $V^T\! A V$ diverges unless the orthogonalization backbone is unconditionally stable.  Otherwise SCIL and MCIL behave the same.  Given that MCIL and Lanczos guarantee shorter recurrences, MCIL-Lanczos would be the winner for this example.

\begin{figure}[htbp!]
    \centering
    \includegraphics[width=0.5\textwidth]{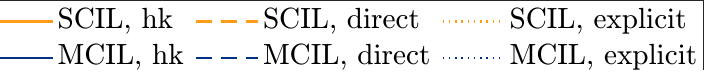}\\[1pt]
    \begin{tabular}{@{}c@{\hspace{2pt}}c@{\hspace{2pt}}c@{}}
        \includegraphics[width=0.31\textwidth]{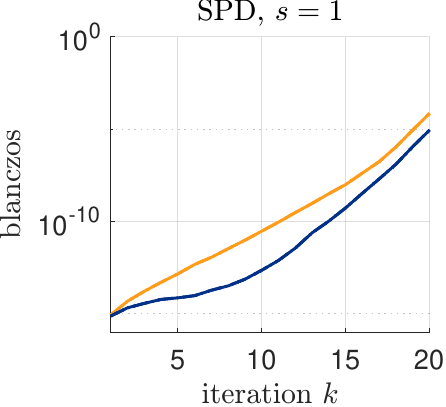} &
        \includegraphics[width=0.31\textwidth]{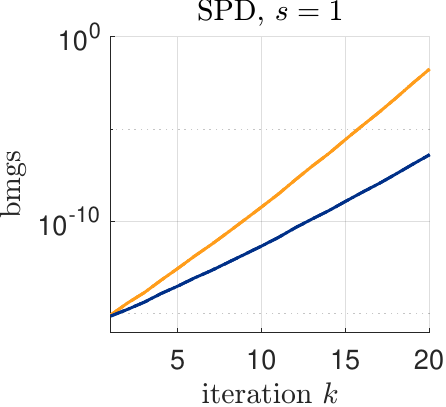} &
        \includegraphics[width=0.31\textwidth]{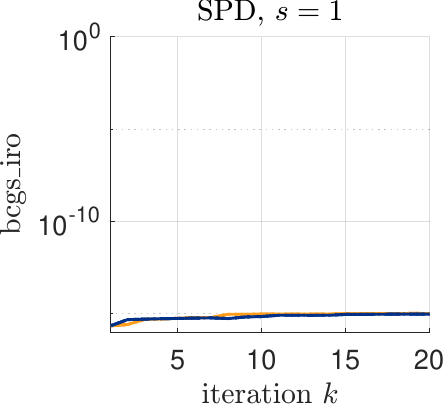} \\
        \includegraphics[width=0.31\textwidth]{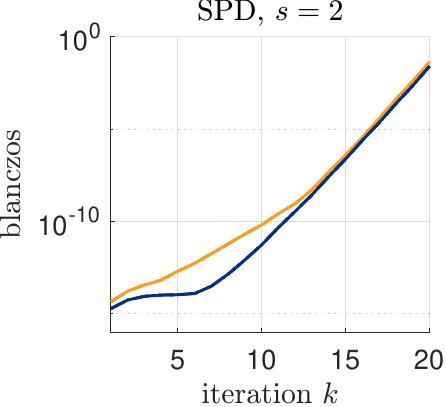} &
        \includegraphics[width=0.31\textwidth]{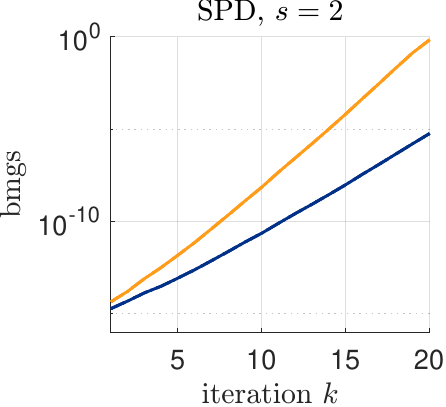} &
        \includegraphics[width=0.31\textwidth]{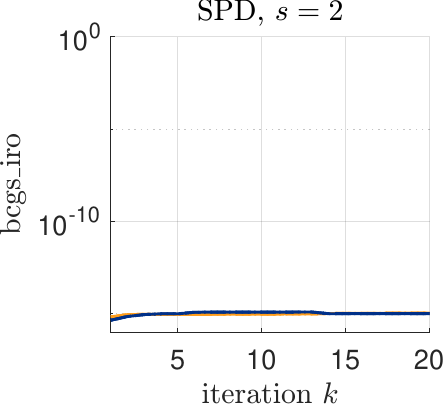} \\
    \end{tabular}
    \caption{LOO $\|I-\vV_k^T\vV_k\|$ for $\exp(A_1)$ (SPD).  See \Cref{ex:expm_spd}.}
    \label{fig:expm_loo_spd}
\end{figure}
\begin{figure}[htbp!]
    \centering
    \includegraphics[width=0.5\textwidth]{fig/expm/legend.pdf}\\[1pt]
    \begin{tabular}{@{}c@{\hspace{2pt}}c@{\hspace{2pt}}c@{}}
        \includegraphics[width=0.31\textwidth]{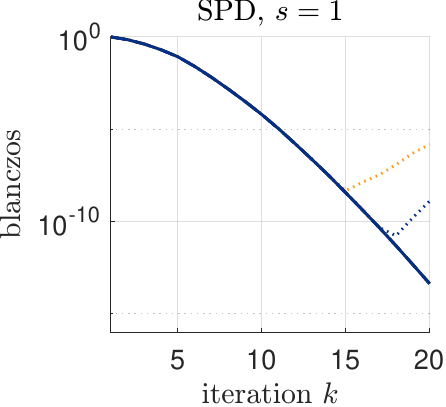} &
        \includegraphics[width=0.31\textwidth]{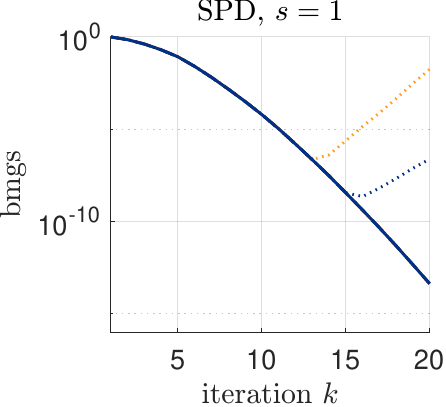} &
        \includegraphics[width=0.31\textwidth]{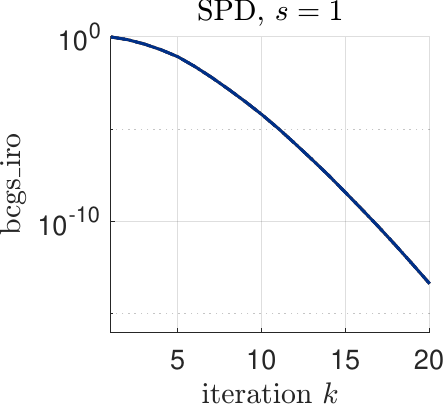} \\
        \includegraphics[width=0.31\textwidth]{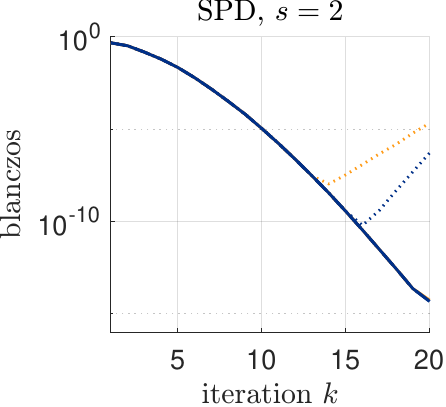} &
        \includegraphics[width=0.31\textwidth]{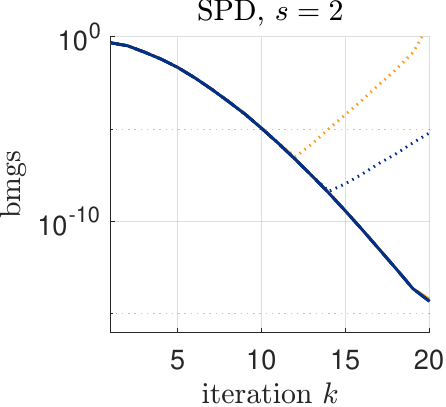} &
        \includegraphics[width=0.31\textwidth]{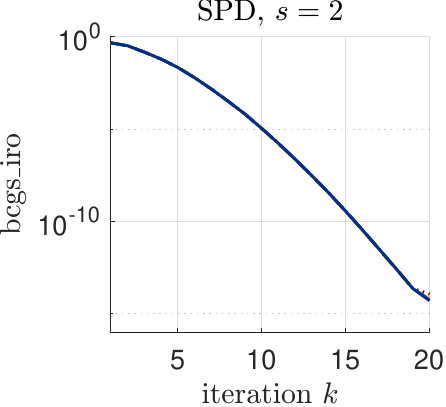} \\
    \end{tabular}
    \caption{True error $\|\fAb-\vx_k\|/\|\fAb\|$ for $\exp(A_1)$ (SPD).  See \Cref{ex:expm_spd}.}
    \label{fig:expm_true_err_spd}
\end{figure}
\end{example}

\begin{example}[$\exp(A_2) \vb$] \label{ex:expm_indef} In this example, we see both SCIL and MCIL with high LOO for \texttt{blanczos}, regardless of block size, while only MCIL appears to struggle for \texttt{bmgs}; cf.\ \Cref{fig:expm_loo_indef}.  As already noted in \Cref{sec:stability}, the orthogonality behavior on indefinite matrices is unpredictable. However, looking at the error in \Cref{fig:expm_true_err_indef}, we again see that really only the choice of how $T$ is computed affects the convergence behavior.  Yet again we would have to conclude that MCIL-Lanczos, with its lower computational costs, is the best approach.
    
\begin{figure}[htbp!]
    \centering
    \includegraphics[width=0.5\textwidth]{fig/expm/legend.pdf}\\[1pt]
    \begin{tabular}{@{}c@{\hspace{2pt}}c@{\hspace{2pt}}c@{}}
        \includegraphics[width=0.31\textwidth]{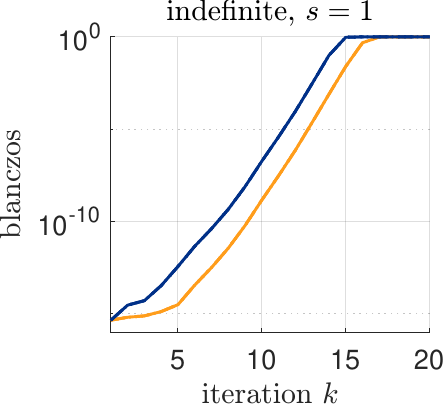} &
        \includegraphics[width=0.31\textwidth]{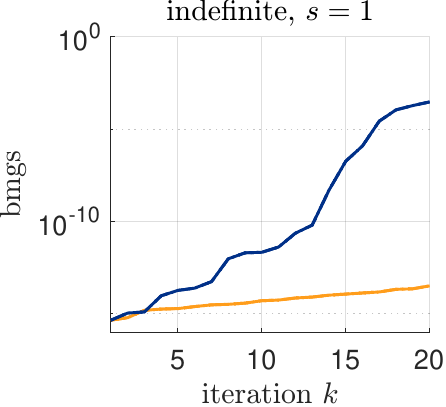} &
        \includegraphics[width=0.31\textwidth]{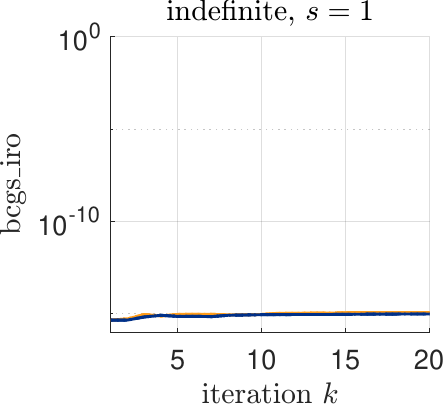} \\
        \includegraphics[width=0.31\textwidth]{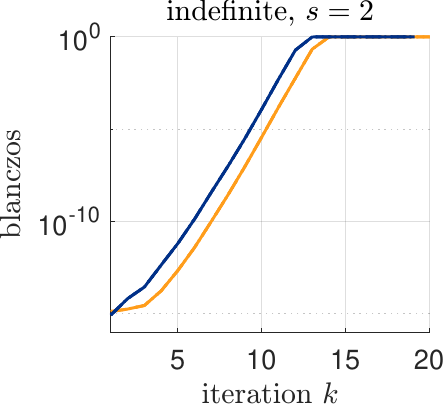} &
        \includegraphics[width=0.31\textwidth]{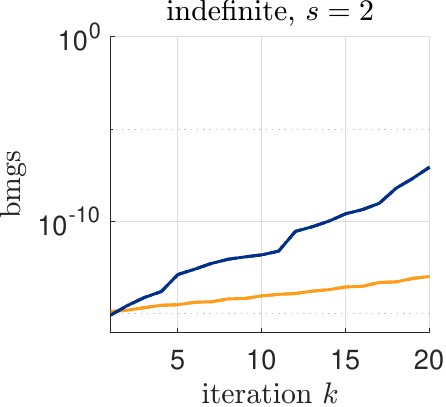} &
        \includegraphics[width=0.31\textwidth]{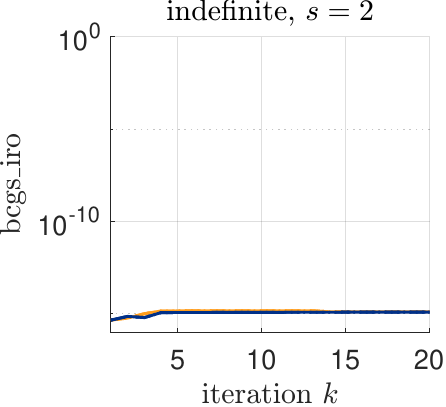} \\
    \end{tabular}
    \caption{LOO $\|I-\vV_k^T\vV_k\|$ for $\exp(A_2)$ (symmetric indefinite).  See \Cref{ex:expm_indef}.}
    \label{fig:expm_loo_indef}
\end{figure}
\begin{figure}[htbp!]
    \centering
    \includegraphics[width=0.5\textwidth]{fig/expm/legend.pdf}\\[1pt]
    \begin{tabular}{@{}c@{\hspace{2pt}}c@{\hspace{2pt}}c@{}}
        \includegraphics[width=0.31\textwidth]{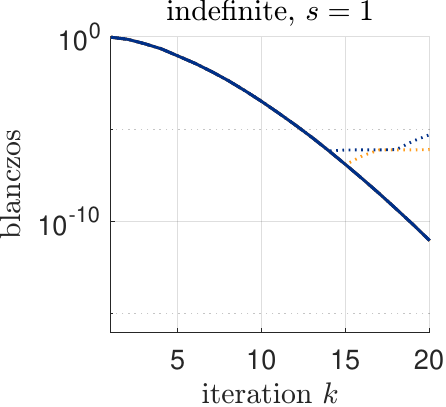} &
        \includegraphics[width=0.31\textwidth]{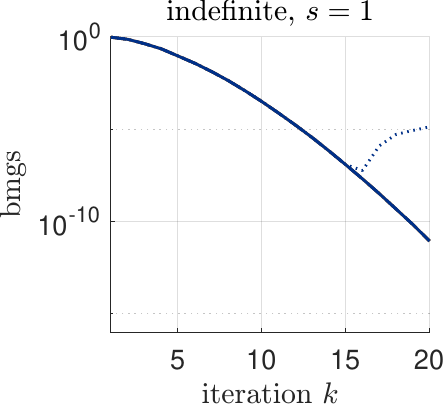} &
        \includegraphics[width=0.31\textwidth]{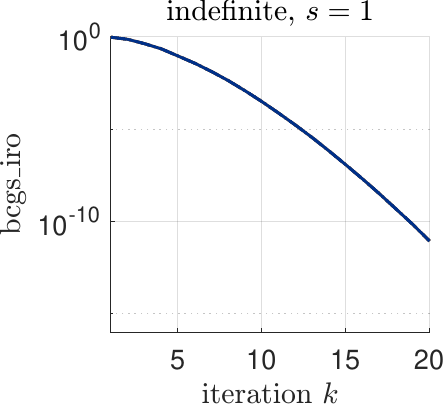} \\
        \includegraphics[width=0.31\textwidth]{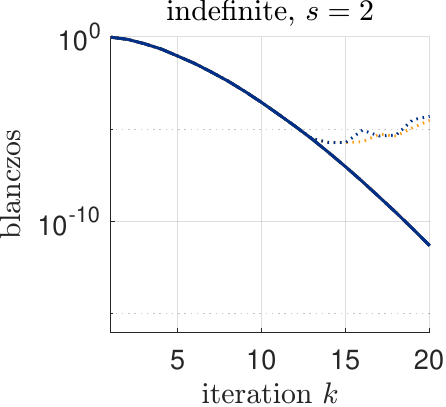} &
        \includegraphics[width=0.31\textwidth]{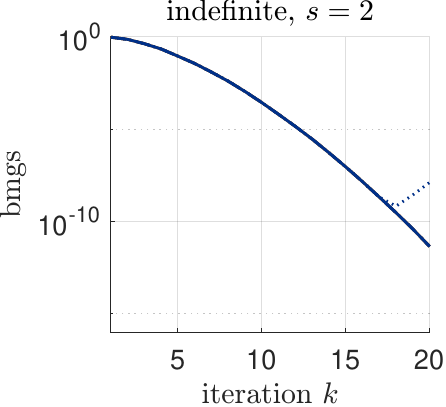} &
        \includegraphics[width=0.31\textwidth]{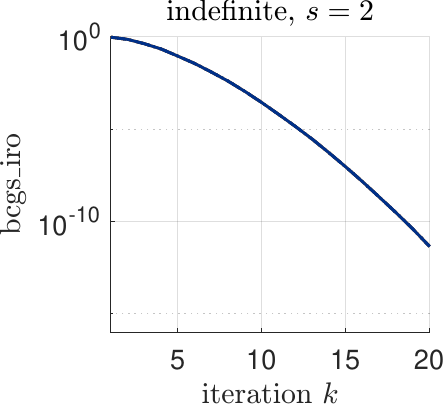} \\
    \end{tabular}
    \caption{True error $\|\fAb-\vx_k\|/\|\fAb\|$ for $\exp(A_2)$ (symmetric indefinite).  See \Cref{ex:expm_indef}.}
    \label{fig:expm_true_err_indef}
\end{figure}
\end{example}

\begin{example}[$\exp(A_3) \vb$] \label{ex:expm_nonsymm}
The plot thickens rather dramatically now.  \Cref{fig:expm_loo_nonsymm} shows that all Lanczos methods completely lose orthogonality almost immediately, while MCIL with \texttt{bmgs} loses orthogonality rather modestly against SCIL with \texttt{bmgs}.  The convergence plots in \Cref{fig:expm_true_err_nonsymm} run counter to previous lessons learned: MCIL with \texttt{blanczos} seemingly converges -- even for block vectors, albeit more slowly -- while all other Lanczos methods fail or stagnate.  For MCIL with \texttt{bmgs} or \texttt{bcgs\_iro}, the \texttt{explicit} $T$ method appears to have a stabilizing effect, even though the convergence is far slower than for any SCIL method.  In the absence of any understanding about MCIL's behavior here, we would recommend SCIL as the more reliable method in this context; to avoid high orthogonalization costs, a low-synchronization Gram-Schmidt method would be best, such as \texttt{bmgs\_svl}~\cite{Bar19}.  There is no reason a Lanczos-type method should lead to predictable convergence on a nonsymmetric matrix, but we cannot conclude this is due to pure chance; see \Cref{ex:invsqrtm_nonsymm}.
    
\begin{figure}[htbp!]
    \centering
    \includegraphics[width=0.5\textwidth]{fig/expm/legend.pdf}\\[1pt]
    \begin{tabular}{@{}c@{\hspace{2pt}}c@{\hspace{2pt}}c@{}}
        \includegraphics[width=0.31\textwidth]{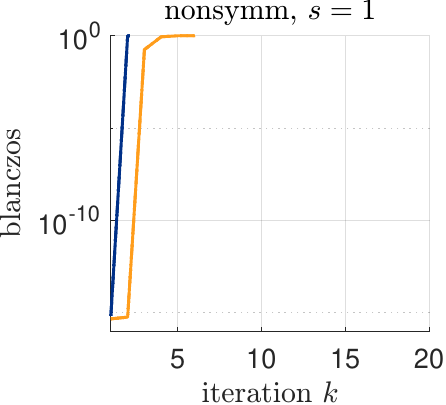} &
        \includegraphics[width=0.31\textwidth]{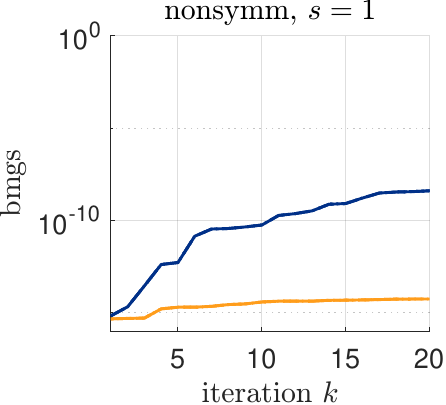} &
        \includegraphics[width=0.31\textwidth]{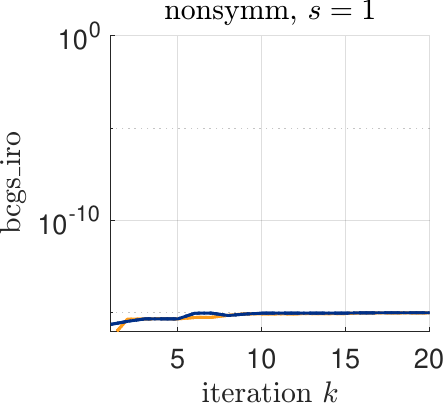} \\
        \includegraphics[width=0.31\textwidth]{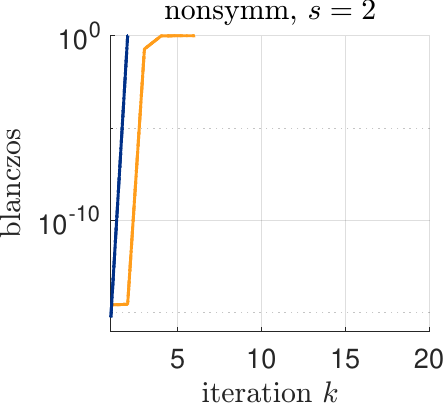} &
        \includegraphics[width=0.31\textwidth]{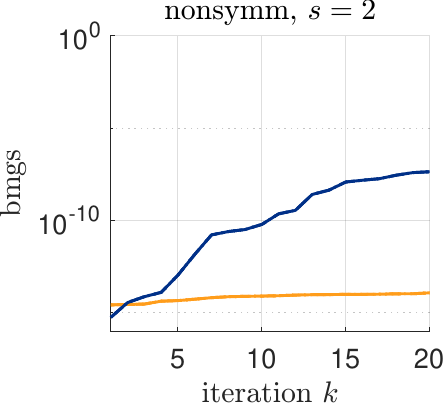} &
        \includegraphics[width=0.31\textwidth]{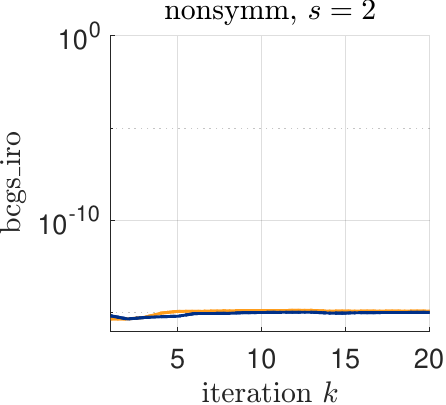} \\
    \end{tabular}
    \caption{LOO $\|I-\vV_k^T\vV_k\|$ for $\exp(A_3)$ (nonsymmetric).  See \Cref{ex:expm_nonsymm}.}
    \label{fig:expm_loo_nonsymm}
\end{figure}
\begin{figure}[htbp!]
    \centering
    \includegraphics[width=0.5\textwidth]{fig/expm/legend.pdf}\\[1pt]
    \begin{tabular}{@{}c@{\hspace{2pt}}c@{\hspace{2pt}}c@{}}
        \includegraphics[width=0.31\textwidth]{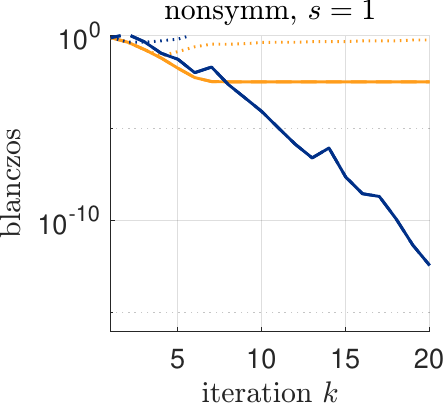} &
        \includegraphics[width=0.31\textwidth]{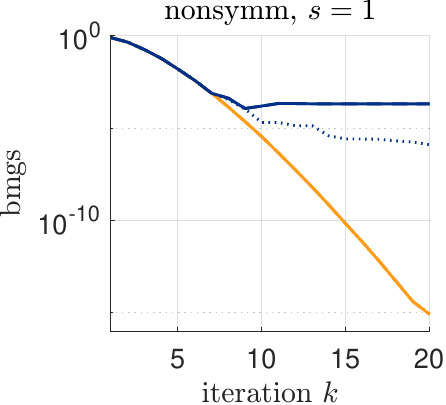} &
        \includegraphics[width=0.31\textwidth]{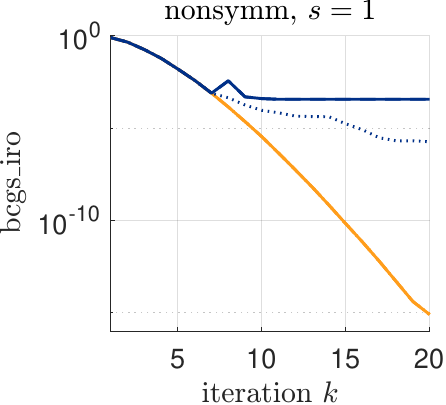} \\
        \includegraphics[width=0.31\textwidth]{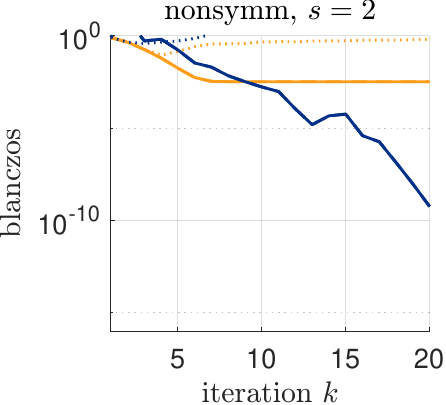} &
        \includegraphics[width=0.31\textwidth]{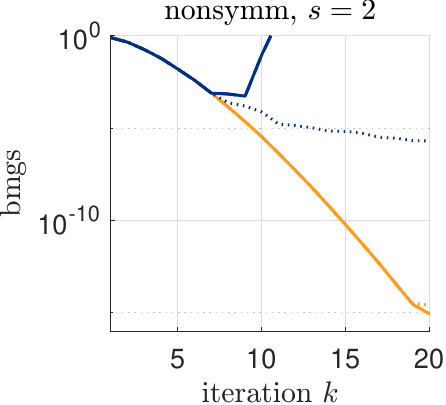} &
        \includegraphics[width=0.31\textwidth]{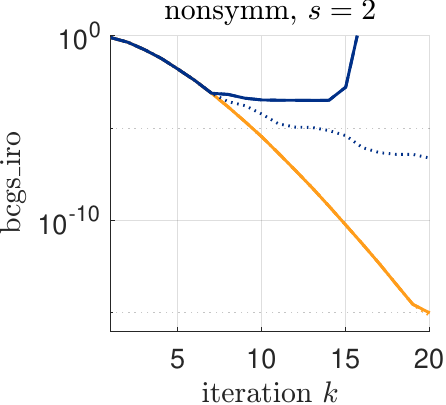} \\
    \end{tabular}
    \caption{True error $\|\fAb-\vx_k\|/\|\fAb\|$ for $\exp(A_3)$ (nonsymmetric).  See \Cref{ex:expm_nonsymm}.}
    \label{fig:expm_true_err_nonsymm}
\end{figure}
\end{example}

\begin{example}[$A_1^{-1/2}\vb$] \label{ex:invsqrtm_spd}
We omit LOO plots for $f(z) = z^{-1/2}$ on SPD $A$, as they are identical to those from \Cref{fig:expm_loo_spd}.  \Cref{fig:invsqrtm_true_err_spd} displays convergence plots for the new function.  Similarly to \Cref{fig:expm_true_err_spd}, we see that all methods converge except for those based on the \texttt{explicit} approach for $T$ when the orthogonalization routine is not accurate to machine precision.
    
\begin{figure}[htbp!]
    \centering
    \includegraphics[width=0.5\textwidth]{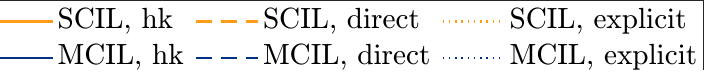}\\[1pt]
    \begin{tabular}{@{}c@{\hspace{2pt}}c@{\hspace{2pt}}c@{}}
        \includegraphics[width=0.31\textwidth]{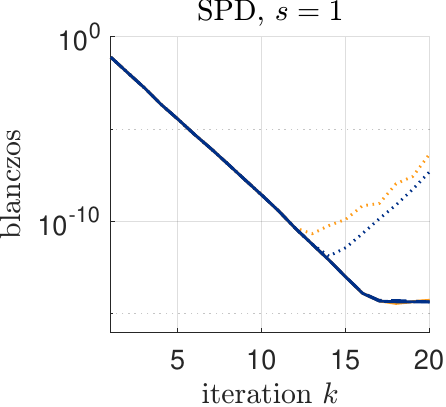} &
        \includegraphics[width=0.31\textwidth]{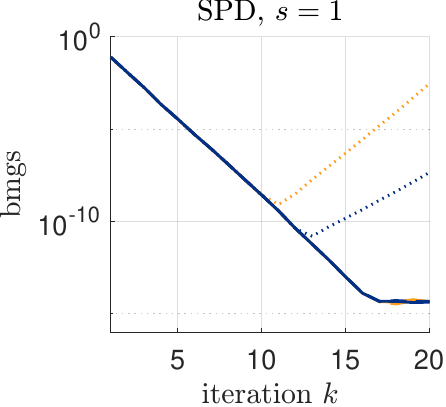} &
        \includegraphics[width=0.31\textwidth]{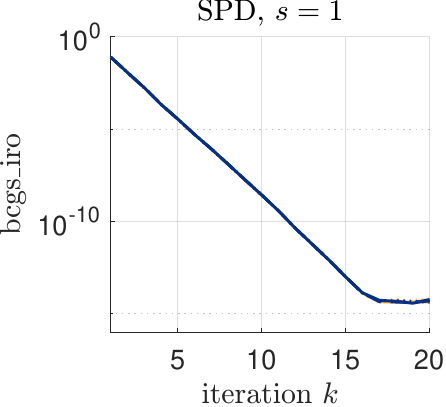} \\
        \includegraphics[width=0.31\textwidth]{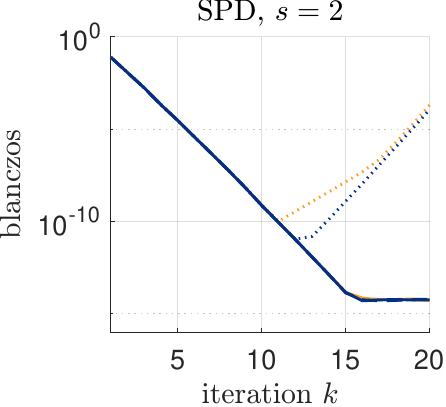} &
        \includegraphics[width=0.31\textwidth]{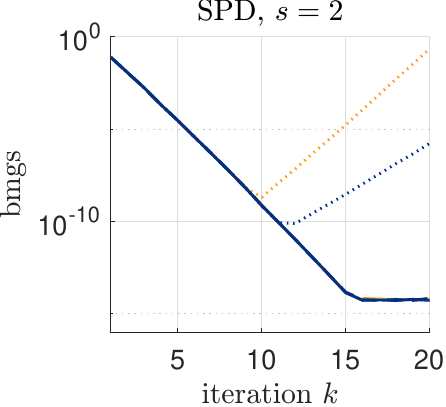} &
        \includegraphics[width=0.31\textwidth]{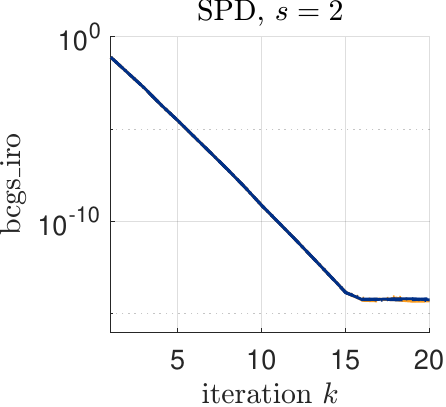} \\
    \end{tabular}
    \caption{True error $\|\fAb-\vx_k\|/\|\fAb\|$ for $A_1^{-1/2}$ (SPD). See \Cref{ex:invsqrtm_spd}.}
    \label{fig:invsqrtm_true_err_spd}
\end{figure}
\end{example}

\begin{example}[$A_4^{-1/2}\vb$] \label{ex:invsqrtm_nonsymm}
Our final example further confounds the story for nonsymmetric matrices.  Recall the differences between $A_4$ and $A_3$; $A_4$ is tailored to avoid branch cuts in the matrix square root.  \Cref{fig:invsqrtm_loo_nonsymm} shows that the LOO for Lanczos-based methods is similar to what we see in \Cref{fig:expm_loo_nonsymm}; but for \texttt{bmgs}, SCIL and MCIL trade places, with SCIL losing more orthogonality.  The convergence plots in \Cref{fig:invsqrtm_true_err_nonsymm} show that yet again, MCIL with \texttt{blanczos} can anyway converge.  Otherwise, the LOO seems only to have an effect when $T$ is computed via the \texttt{explicit} option, and all other non-Lanczos methods converge.

\begin{figure}[htbp!]
    \centering
    \includegraphics[width=0.5\textwidth]{fig/invsqrtm/legend.pdf}\\[1pt]
    \begin{tabular}{@{}c@{\hspace{2pt}}c@{\hspace{2pt}}c@{}}
        \includegraphics[width=0.31\textwidth]{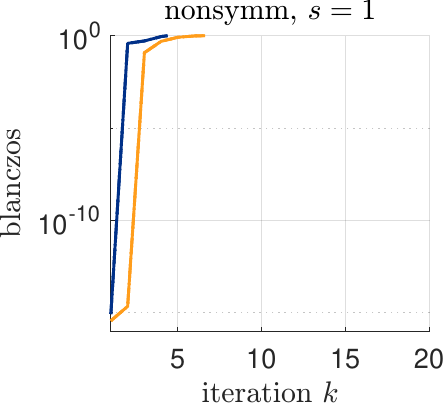} &
        \includegraphics[width=0.31\textwidth]{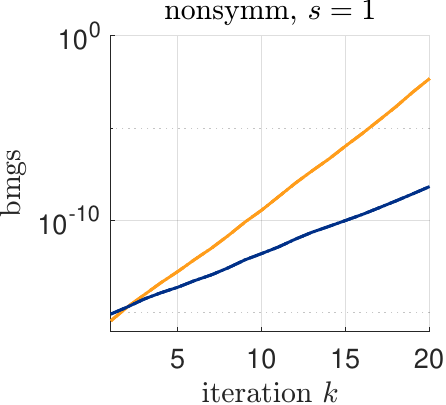} &
        \includegraphics[width=0.31\textwidth]{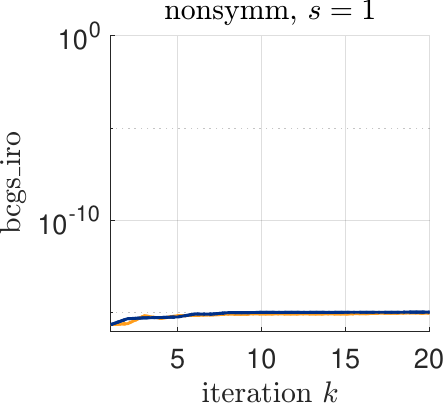} \\
        \includegraphics[width=0.31\textwidth]{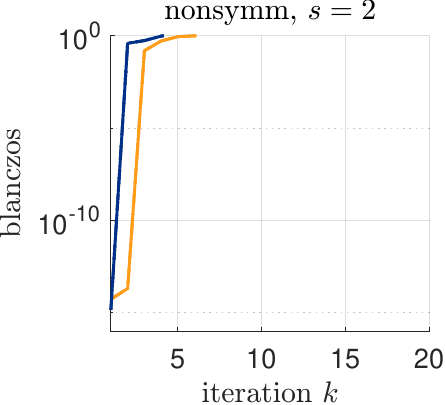} &
        \includegraphics[width=0.31\textwidth]{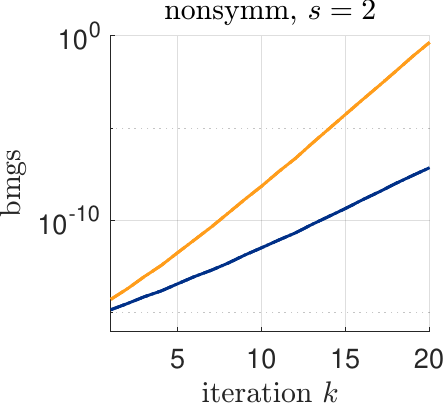} &
        \includegraphics[width=0.31\textwidth]{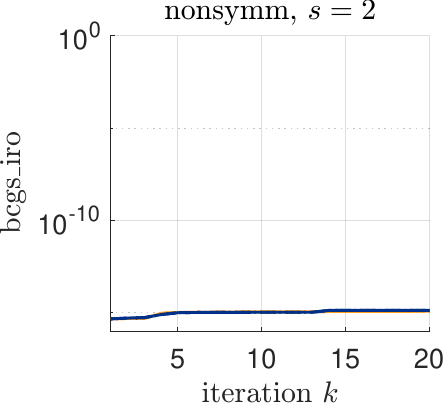} \\
    \end{tabular}
    \caption{LOO $\|I-\vV_k^T\vV_k\|$ for $A_4^{-1/2}$ (nonsymmetric). See \Cref{ex:invsqrtm_nonsymm}.}
    \label{fig:invsqrtm_loo_nonsymm}
\end{figure}
\begin{figure}[htbp!]
    \centering
    \includegraphics[width=0.5\textwidth]{fig/invsqrtm/legend.pdf}\\[1pt]
    \begin{tabular}{@{}c@{\hspace{2pt}}c@{\hspace{2pt}}c@{}}
        \includegraphics[width=0.31\textwidth]{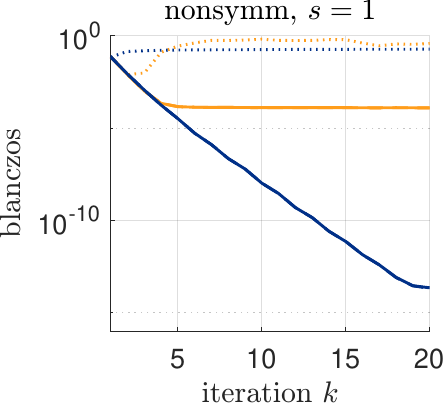} &
        \includegraphics[width=0.31\textwidth]{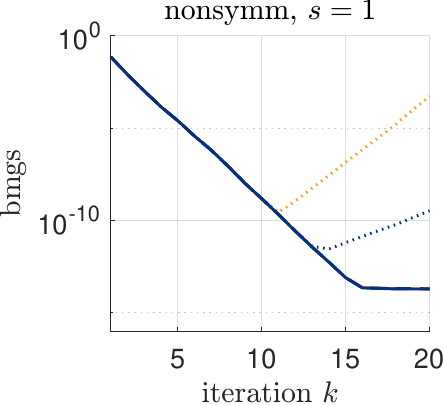} &
        \includegraphics[width=0.31\textwidth]{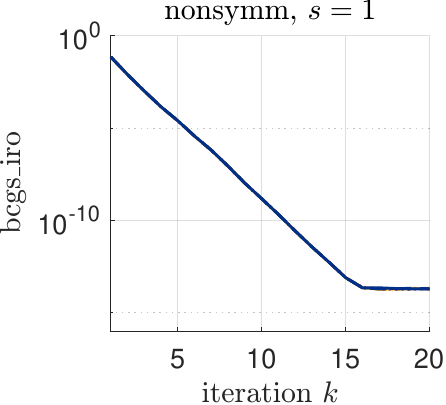} \\
        \includegraphics[width=0.31\textwidth]{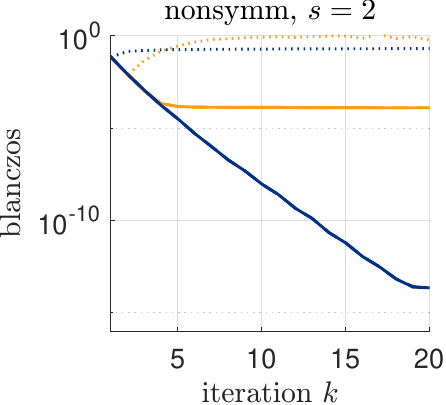} &
        \includegraphics[width=0.31\textwidth]{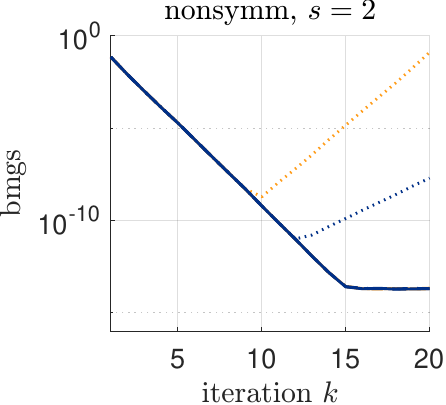} &
        \includegraphics[width=0.31\textwidth]{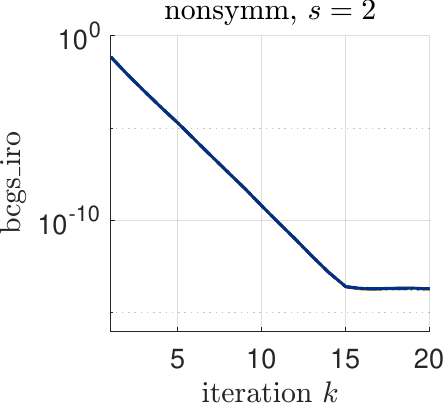} \\
    \end{tabular}
    \caption{True error $\|\fAb-\vx_k\|/\|\fAb\|$for $A_4^{-1/2}$ (nonsymmetric). See \Cref{ex:invsqrtm_nonsymm}.}
    \label{fig:invsqrtm_true_err_nonsymm}
\end{figure}
\end{example}

\section{Conclusions} \label{sec:conclusions}

This work has begun to investigate how basis-building choices in extended Krylov subspace methods affect loss of orthogonality and attainable accuracy in floating-point arithmetic.  We have classified methods by how new basis vectors are added, via either split or mixed channels.  For symmetric positive definite $A$, we have proven that mixed channels lead to 3-term recurrences, while split channels lead to five, and that a simple ratio confirms that mixed channel methods should have better orthogonality than split channel ones.  A diverse panel of numerical experiments, including symmetric indefinite and nonsymmetric matrices, as well as SPD ones, reveals more nuanced behavior.  The starting vector, choice of $T$ method, and orthogonalization backbone have strong effects.  When in doubt, one should ensure the basis is fully orthogonalized to reduce the strength of these effects.

One open question deserves to be highlighted again: mixed channels with a Lanczos backbone appear to converge for some nonsymmetric problems.  We have only tested this on two matrices with two functions; we expect the modularly written \texttt{ektoolbox} to be helpful in further explorations.

\appendix
\crefalias{section}{appendix}
\section{Technical lemmas and proof of \Cref{thm:spd_mcil_vs_scil}} \label{appendix:proofs}
To compare the SCIL and MCIL candidate vectors, we first relate the coefficients in their short recurrences. To do so, we define the auxiliary quantities
\begin{equation*} \label{eq:tau}
    \tau_j := -\frac{h^\S_{2j-2,2j-2}}{h^\S_{2j-1,2j-2}} \text{ for $j \geq 2$}
\end{equation*}
and
\begin{equation} \label{eq:sigma}
    \sigma_j := -\frac{k^\S_{2j-1,2j-1}}{k^\S_{2j,2j-1}} \text{ for $j \geq 1$}.
\end{equation}

\begin{lemma} \label{lem:coefficient-relations}
Let $A$ be symmetric. Then, for $j \geq 2$, the coefficients in the matrices $K^\S$ and $K^\M$ from \Cref{alg:mcil-ek-arnoldi} and \Cref{alg:scil-ek-arnoldi} satisfy
    \begin{equation} \label{eq:inverse-transform}
        k^\M_{2j-1,2j-1} = \tau_j k^\S_{2j-1,2j-1}, \quad k^\M_{2j,2j-1} = \tau_j k^\S_{2j,2j-1}, \quad k^\M_{2j-2,2j-1} = k^\S_{2j-1,2j-1}
    \end{equation}
    For $j=2$, we additionally have
    \begin{equation} \label{eq:inverse-boundary}
        k^\S_{1,3}=k^\S_{2,1},
    \end{equation}
    while for $j\geq3$, we additionally have
    \begin{equation} \label{eq:inverse-outer}
        k^\S_{2j-4,2j-1} = k^\S_{2j-2,2j-3}, \qquad k^\S_{2j-3,2j-1}=\tau_{j-1}k^\S_{2j-4,2j-1}.
    \end{equation}
    
    For $j\geq1$, the coefficients in the matrices $H^\S$ and $H^\M$ from \Cref{alg:mcil-ek-arnoldi} and \Cref{alg:scil-ek-arnoldi} satisfy
    \begin{equation} \label{eq:power-transform}
        h^\M_{2j,2j} = \sigma_j h^\S_{2j,2j}, \qquad h^\M_{2j+1,2j} = \sigma_j h^\S_{2j+1,2j}, \qquad h^\M_{2j-1,2j} = h^\S_{2j,2j}.
    \end{equation}
    For $j\geq2$, we additionally have
    \begin{equation} \label{eq:power-outer}
        h^\S_{2j-3,2j} = h^\S_{2j-1,2j-2}, \qquad h^\S_{2j-2,2j} = \sigma_{j-1}h^\S_{2j-3,2j}.
    \end{equation}
\end{lemma}

\begin{proof}
    We begin with the relations involving the entries of $K^\S$ and $K^\M$. For $j \geq 2$, applying $A^\inv$ to the five-term recurrence~\eqref{eq:short_recurrences_scil} for $A\vv_{2j-3}$ yields
    \begin{equation} \label{eq:apply-inverse-expanded}
    \vv_{2j-3} = \sum_{i=2j-5}^{2j-1} h^\S_{i,2j-2} A^\inv \vv_i.
    \end{equation}
    Only the terms with $i=2j-2$ and $i=2j-1$ on the right-hand side can have a component in the direction of $\vv_{2j}$.  From the recurrences~\eqref{eq:short_recurrences_mcil} for $A^\inv \vv_{2j-1}$ and~\eqref{eq:short_recurrences_scil} for $A^\inv \vv_{2j-2}$, we have
    \begin{equation*}
        A^\inv \vv_{2j-2} = \sum_{i=2j-4}^{2j} k^\S_{i,2j-1} \vv_i
        \quad\text{and}\quad
        A^\inv \vv_{2j-1} = \sum_{i=2j-2}^{2j} k^\M_{i,2j-1} \vv_i.
    \end{equation*}
    As $\vv_{2j-3}$ cannot have a component in direction $\vv_{2j}$, the corresponding coefficients in~\eqref{eq:apply-inverse-expanded} must equate to zero, i.e.,
    \begin{equation} \label{eq:compare-v2j-inverse}
    h^\S_{2j-2,2j-2}k^\S_{2j,2j-1} + h^\S_{2j-1,2j-2}k^\M_{2j,2j-1} = 0.
    \end{equation}
    Similarly, equating the coefficient of $\vv_{2j-1}$ gives
    \begin{equation} \label{eq:compare-v2jm1-inverse}
        h^\S_{2j-2,2j-2} k^\S_{2j-1,2j-1} + h^\S_{2j-1,2j-2} k^\M_{2j-1,2j-1} = 0.
    \end{equation}
    Dividing~\eqref{eq:compare-v2j-inverse} and~\eqref{eq:compare-v2jm1-inverse} by $h^\S_{2j-1,2j-2}$ and inserting the definition~\eqref{eq:tau} proves the first two equalities in~\eqref{eq:inverse-transform}.
    
    The remaining identities involving $K^\S$ and $K^\M$ follow from symmetry: For $j \geq 2$, we have
    \begin{equation*}
        k^\M_{2j-2,2j-1}
        = \ip{\vv_{2j-2}}{A^\inv \vv_{2j-1}}
        = \ip{\vv_{2j-1}}{A^\inv \vv_{2j-2}}
        = k^\S_{2j-1,2j-1},
    \end{equation*}
    which proves the last equality in~\eqref{eq:inverse-transform}. When $j = 2$, the same argument gives
    \begin{equation*}
        k^\S_{1,3}
        = \ip{\vv_1}{A^\inv \vv_2}
        = \ip{\vv_2}{A^\inv \vv_1}
        = k^\S_{2,1},
    \end{equation*}
    and hence~\eqref{eq:inverse-boundary}.  For $j\geq3$, it instead gives
    \begin{equation} \label{eq:inverse-symmetry-outer}
        k^\S_{2j-4,2j-1}
        =\ip{\vv_{2j-4}}{A^\inv \vv_{2j-2}}
        =\ip{\vv_{2j-2}}{A^\inv \vv_{2j-4}}
        =k^\S_{2j-2,2j-3}.
    \end{equation}
    Moreover,
    \begin{align*}
        k^\S_{2j-3,2j-1}
        &= \ip{\vv_{2j-3}}{A^\inv \vv_{2j-2}} = \ip{\vv_{2j-2}}{A^\inv \vv_{2j-3}} \\
        &= k^\M_{2j-2,2j-3} = \tau_{j-1} k^\S_{2j-2,2j-3} = \tau_{j-1} k^\S_{2j-4,2j-1},
    \end{align*}
    where the second-to-last equality results from~\eqref{eq:inverse-transform} with $j-1$ in place of $j$, and the last equality follows from~\eqref{eq:inverse-symmetry-outer}. This proves~\eqref{eq:inverse-outer}.
    
    We now turn to the relations involving $H^\S$ and $H^\M$.  For $j \geq 2$, we apply $A$ to the five-term recurrence~\eqref{eq:short_recurrences_scil} for $A^\inv \vv_{2j-2}$, which yields
    \begin{equation*}
        \vv_{2j-2} = \sum_{i=2j-4}^{2j} k^\S_{i,2j-1} A \vv_i.
    \end{equation*}
    Similarly to before, only the terms with $i=2j-1$ and $i=2j$ can have components in the directions of $\vv_{2j}$ and $\vv_{2j+1}$, and their recurrences are
    \[
        A \vv_{2j-1} = \sum_{i=2j-3}^{2j+1} h^\S_{i,2j} \vv_i
        \quad\text{and}\quad
        A \vv_{2j} = \sum_{i=2j-1}^{2j+1} h^\M_{i,2j} \vv_i.
    \]
    Noting that $\vv_{2j-2}$ cannot have components in direction $\vv_{2j+1}$ or $\vv_{2j}$ and equating the corresponding coefficients yields
    \begin{equation*}
        k^\S_{2j-1,2j-1} h^\S_{2j+1,2j} + k^\S_{2j,2j-1} h^\M_{2j+1,2j} = 0
    \end{equation*}
    and
    \begin{equation*}
        k^\S_{2j-1,2j-1} h^\S_{2j,2j} + k^\S_{2j,2j-1} h^\M_{2j,2j} = 0.
    \end{equation*}
    Dividing these equations by $k^\S_{2j,2j-1}$ and inserting~\eqref{eq:sigma} proves the first two equalities in~\eqref{eq:power-transform} for $j\geq2$. For $j=1$, applying $A$ to the first inverse recurrence gives
    \[
        \vv_1 = k^\S_{11} A \vv_1 + k^\S_{21} A \vv_2
    \]
    and comparison of the coefficients of $\vv_3$ and $\vv_2$, respectively, gives
    \[
        k^\S_{11}h^\S_{32}+k^\S_{21}h^\M_{32} = 0, \quad\text{ and }\quad k^\S_{11}h^\S_{22}+k^\S_{21}h^\M_{22} = 0.
    \]
    These are precisely the first two identities in~\eqref{eq:power-transform} for $j=1$.
    
    Finally, symmetry of $A$ gives, for every $j \geq 1$,
    \begin{align*}
        h^\M_{2j-1,2j}
        &= \ip{\vv_{2j-1}}{A \vv_{2j}}
        =\ip{\vv_{2j}}{A \vv_{2j-1}}
        =h^\S_{2j,2j},
    \end{align*}
    which proves the last equality in~\eqref{eq:power-transform}.  For $j \geq 2$, the same argument also yields
    \begin{equation} \label{eq:power-symmetry-outer}
        h^\S_{2j-3,2j}
        =\ip{\vv_{2j-3}}{A \vv_{2j-1}}
        =\ip{\vv_{2j-1}}{A \vv_{2j-3}}
        =h^\S_{2j-1,2j-2}.
    \end{equation}
    The other identity in~\eqref{eq:power-outer} follows from
    \begin{align*}
        h^\S_{2j-2,2j}
        &= \ip{\vv_{2j-2}}{A \vv_{2j-1}} = \ip{\vv_{2j-1}}{A \vv_{2j-2}} \\
        &= h^\M_{2j-1,2j-2} = \sigma_{j-1} h^\S_{2j-1,2j-2} = \sigma_{j-1} h^\S_{2j-3,2j},
    \end{align*}
    where the second-to-last equality follows from~\eqref{eq:power-transform} with $j-1$ in place of $j$, and the last equality follows from~\eqref{eq:power-symmetry-outer}.
\end{proof}

The proof of \Cref{thm:spd_mcil_vs_scil} crucially relies on exploiting the sign pattern of the entries of the matrices $H$ and $K$, which we establish next.

\begin{lemma} \label{lem:signs}
    Let $A$ be symmetric positive definite. Then all nonzero entries in odd columns of $K^\S$ are positive and the nonzero entries in even columns of $H^\S$ have signs
    \begin{equation*}
        h^\S_{2j-3,2j} > 0, \quad
        h^\S_{2j-2,2j} < 0, \quad
        h^\S_{2j-1,2j} > 0, \quad
        h^\S_{2j,2j} < 0, \quad
        h^\S_{2j+1,2j} > 0,
    \end{equation*}
    In particular, $\tau_j > 0$ for $j \geq 2$ and $\sigma_j < 0$ for $j \geq 1$.
\end{lemma}

\begin{proof}
    We prove the result by induction: It is easy to establish that the first column of $K^\S$ has positive entries: Because $A^\inv$ is positive definite, we have $k^\S_{11} = \ip{\vv_1}{A^\inv \vv_1} > 0$. Further, $k^\S_{21} > 0$  as it the norm of the second basis vector candidate after orthogonalization (which cannot be equal to zero, as we assumed no breakdown). Hence $\sigma_1 < 0$. Similarly, we have $h^\S_{12} = \ip{\vv_1}{A \vv_1} > 0$ and $h^\S_{32}>0$. Since
    \[
        h^\M_{22} = \sigma_1 h^\S_{22}
    \]
    by \Cref{lem:coefficient-relations} and $h^\M_{22} = \ip{\vv_2}{A \vv_2} > 0$, it follows that $h^\S_{22} < 0$. This concludes the base case.
    
    Suppose now that the sign pattern has been established through column $2j-2$ of $H^\S$. Then $\tau_j > 0$. The outer entries in column $2j-1$ of $K^\S$ are positive by~\eqref{eq:inverse-boundary} and~\eqref{eq:inverse-outer}. We have
    \[
        k^\S_{2j-2,2j-1} = \ip{\vv_{2j-2}}{A^\inv \vv_{2j-2}} > 0,
    \]
    as well as $k^\S_{2j,2j-1} > 0$, because it is the norm of a nonzero vector. Finally,
    \[
        k^\S_{2j-1,2j-1} = \frac{k^\M_{2j-1,2j-1}}{\tau_j} > 0,
    \]
    because $k^\M_{2j-1,2j-1} = \ip{\vv_{2j-1}}{A^\inv \vv_{2j-1}}>0$. Thus the new odd column of $K^\S$ is positive and $\sigma_j < 0$.
    
    It remains to investigate the nonzero entries of the next even column of $H^\S$. By~\eqref{eq:power-outer}, we directly have $h^\S_{2j-3,2j} > 0$ and $h^\S_{2j-2,2j} < 0$. The next entry is the Rayleigh quotient
    \[
        h^\S_{2j-1,2j} = \ip{\vv_{2j-1}}{A \vv_{2j-1}} > 0.
    \]
    Moreover,
    \[
        h^\S_{2j,2j} = \frac{h^\M_{2j,2j}}{\sigma_j} < 0,
    \]
    because $h^\M_{2j,2j} = \ip{\vv_{2j}}{A \vv_{2j}} > 0$, and $h^\S_{2j+1,2j} > 0$ is the norm of a nonzero vector, which completes the induction step and thus proves the result.
\end{proof}

As a last ingredient, we need the following scalar inequality result.

\begin{lemma} \label{lem:scalar-comparison}
    Let $a, c, x > 0$, let $b \in \spR$, and suppose
    \[
        x = c + \mu a, \qquad \mu < 0.
    \]
    Then
    \[
        a^2 + b^2 + c^2 - x^2 > 0.
    \]
\end{lemma}

\begin{proof}
    Since $x > 0$, one has $c > -\mu a$.  Therefore,
    \[
        a^2 + b^2 + c^2 - x^2
        = (1 - \mu^2) a^2 + b^2 - 2 \mu ac
        > (1 + \mu^2) a^2 + b^2
        > 0.
    \]
\end{proof}

We are now in a position to prove \Cref{thm:spd_mcil_vs_scil}.

\begin{proof}[Proof of \Cref{thm:spd_mcil_vs_scil}]
    The candidate vectors at the second step coincide, so we clearly have $\rho_2^\M = \rho_2^\S$.
    
    We first consider the even steps. The case $j = 2$, corresponding to step $4$, has to be treated separately. To lighten the notation, we introduce the shorthands
    \[
        a := k^\S_{1,3}, \qquad
        c := k^\S_{2,3}, \qquad
        d := k^\S_{3,3}, \qquad
        e := k^\S_{4,3}, \qquad
        \mu_2 := \frac{h^\S_{1,2}}{h^\S_{2,2}}.
    \]
    By \Cref{lem:signs}, $a, c, d, e > 0$ and $\mu_2 < 0$. Moreover,~\eqref{eq:inverse-boundary} gives $k^\S_{2,1} = a$, while~\eqref{eq:inverse-transform} with $j = 2$ gives $k^\M_{2,3} = d$, $k^\M_{3,3} = \tau_2d$, and $k^\M_{4,3} = \tau_2e$. The relevant recurrences can therefore be written as
    \begin{align*}
        A \vv_1
        &= h^\S_{1,2} \vv_1 + h^\S_{2,2} \vv_2 + h^\S_{3,2} \vv_3,
        & A^\inv \vv_1
        &= k^\S_{11} \vv_1 + a \vv_2,\\
        A^\inv \vv_2
        &= a \vv_1 + c \vv_2 + d \vv_3 + e \vv_4,
        & A^\inv \vv_3
        &= d \vv_2 + \tau_2d \vv_3 + \tau_2 e \vv_4.
    \end{align*}
    Equating the coefficient of $\vv_2$ in $\vv_1 = A^\inv (A \vv_1)$ yields
    \[
        h^\S_{1,2} a + h^\S_{2,2} c + h^\S_{3,2} d = 0.
    \]
    Since $\tau_2=-h^\S_{2,2}/h^\S_{3,2}$, we obtain
    \begin{equation} \label{eq:even-boundary-reduction}
        \frac{d}{\tau_2} = c + \mu_2 a.
    \end{equation}
    For SCIL, the candidate vector at step $4$ is $A^\inv \vv_2$, and the vector remaining after its orthogonalization is $e \vv_4$. For MCIL, the candidate vector is $A^\inv \vv_3$, and the vector remaining after its orthogonalization is $\tau_2 e \vv_4$. It follows that
    \[
        (\rho_4^\S)^{-2} = 1 + \frac{a^2 + c^2 + d^2}{e^2}, \qquad (\rho_4^\M)^{-2} = 1 + \frac{(d/\tau_2)^2 + d^2}{e^2}.
    \]
    Applying \Cref{lem:scalar-comparison} to~\eqref{eq:even-boundary-reduction}, with $x = d/\tau_2$ and $b = 0$, gives
    \[
        (\rho_4^\S)^{-2} - (\rho_4^\M)^{-2} = \frac{a^2 + c^2 - (d/\tau_2)^2}{e^2} > 0.
    \]
    Consequently, $\rho_4^\M > \rho_4^\S$.
    
    We next consider an even step $2j$ with $j \geq 3$. To simplify the notation, we introduce 
    \[
        a := k^\S_{2j-4,2j-1}, \,\, 
        b := k^\S_{2j-3,2j-1}, \,\,
        c := k^\S_{2j-2,2j-1}, \,\,
        d := k^\S_{2j-1,2j-1}, \,\,
        e := k^\S_{2j,2j-1}.
    \]
    The SCIL recurrence~\eqref{eq:short_recurrences_scil} and~\eqref{eq:inverse-transform} give
    \begin{equation} \label{eq:even-candidates}
        \begin{aligned}
            A^\inv \vv_{2j-2} &= a \vv_{2j-4} + b \vv_{2j-3} + c \vv_{2j-2} + d \vv_{2j-1} + e \vv_{2j},\\
            A^\inv \vv_{2j-1} &= d \vv_{2j-2} + \tau_j d \vv_{2j-1} + \tau_j e \vv_{2j}.
        \end{aligned}
    \end{equation}
    Define
    \begin{equation} \label{eq:mu}
        \mu_j = \frac{h^\S_{2j-4,2j-2} + \tau_{j-1} h^\S_{2j-3,2j-2}}{h^\S_{2j-2,2j-2}}.
    \end{equation}
    As in the proof of \Cref{lem:coefficient-relations}, we apply $A^\inv$ to the SCIL recurrence~\eqref{eq:short_recurrences_scil} for $A\vv_{2j-3}$ and equate the coefficient of $\vv_{2j-2}$. By~\eqref{eq:inverse-outer}, the coefficient of $\vv_{2j-2}$ in $A^\inv \vv_{2j-4}$ is $a$. By~\eqref{eq:inverse-transform} at index $j-1$, the corresponding coefficient in $A^\inv \vv_{2j-3}$ is $\tau_{j-1} a$. Finally, the coefficients of $\vv_{2j-2}$ in $A^\inv \vv_{2j-2}$ and $A^\inv \vv_{2j-1}$ are $c$ and $d$, respectively. Hence
    \[
        \bigl( h^\S_{2j-4,2j-2} + \tau_{j-1} h^\S_{2j-3,2j-2} \bigr) a + h^\S_{2j-2,2j-2} c + h^\S_{2j-1,2j-2}d = 0.
    \]
    Using $h^\S_{2j-1,2j-2} = -h^\S_{2j-2,2j-2} / \tau_j$ and~\eqref{eq:mu}, this becomes
    \begin{equation} \label{eq:even-reduction}
        \frac{d}{\tau_j} = c + \mu_j a.
    \end{equation}
    It remains to determine the sign of $\mu_j$. Equating the coefficient of $\vv_{2j-3}$ in the same vector identity gives
    \[
        \bigl( h^\S_{2j-4,2j-2} + \tau_{j-1} h^\S_{2j-3,2j-2} \bigr) k^\S_{2j-3,2j-3} + h^\S_{2j-2,2j-2}b = 1.
    \]
    By \Cref{lem:signs}, we have $k^\S_{2j-3,2j-3} > 0$, as well as $h^\S_{2j-2,2j-2} < 0$, and $b > 0$. Thus, the second term on the left-hand side is negative, so the factor in parentheses is positive. Its quotient with $h^\S_{2j-2,2j-2}$ is therefore negative, i.e., $\mu_j < 0$.
    
    The residual coefficients in the two recurrences in~\eqref{eq:even-candidates} are $e$ and $\tau_j e$, respectively. Consequently,
    \[
        (\rho_{2j}^\S)^{-2} = 1 + \frac{a^2 + b^2 + c^2 + d^2}{e^2}, \qquad
        (\rho_{2j}^\M)^{-2} = 1 + \frac{(d / \tau_j)^2 + d^2}{e^2}.
    \]
    Since $d/\tau_j > 0$, \Cref{lem:scalar-comparison} applied to~\eqref{eq:even-reduction} yields
    \[
    (\rho_{2j}^\S)^{-2} - (\rho_{2j}^\M)^{-2} = \frac{a^2 + b^2 + c^2 - (d/\tau_j)^2}{e^2} > 0,
    \]
    and hence $\rho_{2j}^\M > \rho_{2j}^\S$ for every $j \geq 3$.
    
    We now turn to the odd steps. We separately treat the case $j = 1$, which corresponds to step $3$. Set
    \[
        c = h^\S_{1,2}, \quad
        d = h^\S_{2,2}, \quad
        e = h^\S_{3,2}, \quad
        x = \frac{d}{\sigma_1}.
    \]
    By \Cref{lem:signs}, we have $c > 0$, $e > 0$, $d < 0$ and $\sigma_1 < 0$, so that $x > 0$. The SCIL recurrence~\eqref{eq:short_recurrences_scil} and~\eqref{eq:power-transform} give
    \[
        A \vv_1 = c \vv_1 + d \vv_2 + e \vv_3, \qquad
        A \vv_2 = d \vv_1 + \sigma_1 d \vv_2 + \sigma_1 e \vv_3.
    \]
    Applying $A$ to $A^\inv \vv_1 = k^\S_{11} \vv_1 + k^\S_{21} \vv_2$ and equating the coefficient of $\vv_1$ gives
    \[
        1 = k^\S_{11} c + k^\S_{21} d.
    \]
    Since $\sigma_1 = -k^\S_{11} / k^\S_{21}$, we obtain
    \[
    x = \frac{d}{\sigma_1} = c - \frac{1}{k^\S_{11}} < c.
    \]
    The residual coefficients of the SCIL and MCIL candidate vectors are $e$ and $\sigma_1e$, respectively. Therefore,
    \[
        (\rho_3^\S)^{-2} = 1 + \frac{c^2 + d^2}{e^2}, \qquad
        (\rho_3^\M)^{-2} = 1 + \frac{x^2 + d^2}{e^2},
    \]
    and thus
    \[
    (\rho_3^\S)^{-2} - (\rho_3^\M)^{-2} = \frac{c^2 - x^2}{e^2} > 0.
    \]
    This proves $\rho_3^\M > \rho_3^\S$.
    
    Finally, consider an odd step $2j+1$ with $j \geq 2$, and set
    \[
        a = h^\S_{2j-3,2j}, \quad
        b = h^\S_{2j-2,2j}, \quad
        c = h^\S_{2j-1,2j}, \quad
        d = h^\S_{2j,2j}, \quad
        e = h^\S_{2j+1,2j}
    \]
    as well as
    \[
        q = k^\S_{2j-3,2j-1}, \qquad
        r = k^\S_{2j-2,2j-1}, \qquad
        s = k^\S_{2j-1,2j-1}, \qquad
        t = k^\S_{2j,2j-1}.
    \]
    Thus, $q,r,s,t$ are the coefficients of $\vv_{2j-3}, \vv_{2j-2}, \vv_{2j-1}, \vv_{2j}$, respectively, in the SCIL recurrence~\eqref{eq:short_recurrences_scil} for $A^\inv \vv_{2j-2}$. If $j \geq 3$, this recurrence contains the additional term $k^\S_{2j-4,2j-1} \vv_{2j-4}$, which cannot contribute to either of the coefficient comparisons below. By~\eqref{eq:sigma}, we have $t = -s/\sigma_j$. Moreover, the SCIL recurrence~\eqref{eq:short_recurrences_scil} and~\eqref{eq:power-transform} give
    \begin{equation} \label{eq:odd-candidates}
        \begin{aligned}
            A \vv_{2j-1} &= a \vv_{2j-3} + b \vv_{2j-2} + c \vv_{2j-1} + d \vv_{2j} + e \vv_{2j+1},\\
            A \vv_{2j}   &= d \vv_{2j-1} + \sigma_j d \vv_{2j} + \sigma_j e \vv_{2j+1}.
        \end{aligned}
    \end{equation}
    Define
    \begin{equation} \label{eq:nu}
        \nu_j = \frac{q + \sigma_{j-1} r}{s}.
    \end{equation}
    We apply $A$ to the SCIL recurrence for $A^\inv \vv_{2j-2}$. By~\eqref{eq:power-transform} at index $j-1$ and~\eqref{eq:power-outer},
    \[
        h^\M_{2j-2,2j-2} =\sigma_{j-1} h^\S_{2j-2,2j-2}, \qquad
        h^\S_{2j-1,2j-2} = a, \qquad
        h^\M_{2j-1,2j-2} = \sigma_{j-1} a.
    \]
    Equating the coefficient of $\vv_{2j-2}$ in $\vv_{2j-2} = A (A^\inv \vv_{2j-2})$ gives
    \begin{equation} \label{eq:proof_aux_1}
        (q + \sigma_{j-1} r) h^\S_{2j-2,2j-2} + s b = 1.
    \end{equation}
    By \Cref{lem:signs}, we have $h^\S_{2j-2,2j-2} < 0$, $b < 0$, and $s > 0$. Hence $sb < 0$, so the first term on the left-hand side of~\eqref{eq:proof_aux_1} is positive. It follows that $q + \sigma_{j-1} r < 0$, and therefore $\nu_j < 0$. Equating the coefficient of $\vv_{2j-1}$ in the same vector identity gives
    \[
        (q + \sigma_{j-1} r) a + sc + td = 0.
    \]
    Substituting $t = -s / \sigma_j$ and using~\eqref{eq:nu} yields
    \begin{equation} \label{eq:odd-reduction}
        \frac{d}{\sigma_j} = c + \nu_j a.
    \end{equation}
    By \Cref{lem:signs}, we have $d / \sigma_j > 0$. The residual coefficients in the two recurrences in~\eqref{eq:odd-candidates} are $e$ and $\sigma_je$, respectively. Hence
    \[
        (\rho_{2j+1}^\S)^{-2} = 1 + \frac{a^2 + b^2 + c^2 + d^2}{e^2}, \qquad
        (\rho_{2j+1}^\M)^{-2} = 1 + \frac{(d/\sigma_j)^2 + d^2}{e^2}.
    \]
    Applying \Cref{lem:scalar-comparison} to~\eqref{eq:odd-reduction} gives
    \[
    (\rho_{2j+1}^\S)^{-2} - (\rho_{2j+1}^\M)^{-2} = \frac{a^2 + b^2 + c^2 - (d/\sigma_j)^2}{e^2}>0.
    \]
    Consequently, $\rho_{2j+1}^\M > \rho_{2j+1}^\S$ for every $j \geq 2$, which completes the proof.
\end{proof}

\section*{Acknowledgments}
Lund would like to acknowledge funding from the London Mathematical Society's Research in Pairs program.  This funding covered Lund's visits to the University of Wuppertal in October 2025 and February 2026. The authors also benefited from discussions at meetings of ``The $f(A)b$ulous scientific network on matrix functions and exponential integrators'' funded by DFG grant 566049107.

Lund used Claude Sonnet 4.6 and Opus 4.8 to review code, debug, and create issues in \texttt{ektoolbox}; furthermore, some high-level routines were generated completely by Claude.  All code has been reviewed by humans.  Claude was also used to help with the LaTeX formatting of this document.

\bibliographystyle{siam}
\bibliography{references}
\end{document}